\documentclass{birkjour}

\usepackage{amsmath}
\usepackage{amssymb}
\usepackage{amsthm}
\usepackage{graphicx}   

\newtheorem{thm}{Theorem}
\newtheorem{prop}{Proposition}[section]
\newtheorem{cor}[prop]{Corollary}
\newtheorem{lem}[prop]{Lemma}
\newtheorem{prob}[prop]{Problem}
\newtheorem{defn}[prop]{Definition}
\newtheorem{rem}[prop]{Remark}
\newtheorem{ex}[prop]{Example}

\newcommand{\osymm}{\origo-symmetric }
\renewcommand{\Re}{\mathbb R}
\renewcommand{\epsilon}{\varepsilon}

\newcommand{\B}{\mathbf{B}}

\newcommand{\BB}{\mathcal{B}}

\newcommand{\GG}{\mathcal{G}}
\newcommand{\TT}{\mathcal{T}}
\newcommand{\KK}{\mathcal{K}}

\newcommand{\remark}[1]{}
\newcommand{\Ze}{\mathbb Z}
\newcommand{\origo}{\ensuremath o}
\newcommand{\Bp}[1][C]{\B_{#1}^{+}}
\newcommand{\Bm}[1][C]{\B_{#1}^{-}}
\newcommand{\BC}[1][C]{\B_{#1}}                            

\newcommand{\bconv}[2][C]{{\mathrm{conv}}_{#1}^{b}\left(#2 \right)}
\newcommand{\bconvures}[1][C]{{\mathrm{conv}}_{#1}^{b}}
\newcommand{\sconv}[2][C]{{\mathrm{conv}}_{#1}^{s}\left(#2 \right)}

\newcommand{\Ren}{\Re^n}

\newcommand{\FF}{\mathcal{F}}
\newcommand{\D}[1][C]{\mathcal{D}({#1})}
\newcommand{\DT}[1][C]{\tilde{\mathcal D}({#1})}
\DeclareMathOperator{\inter}{int}
\DeclareMathOperator{\relint}{relint}
\DeclareMathOperator{\bd}{bd}

\DeclareMathOperator{\conv}{conv}
\DeclareMathOperator{\card}{card}
\DeclareMathOperator{\cl}{cl}

\DeclareMathOperator{\proj}{proj}
\DeclareMathOperator{\diam}{diam}
\DeclareMathOperator{\Car}{Car}
\DeclareMathOperator{\arclength}{\rho}
\DeclareMathOperator{\perim}{perim}
\DeclareMathOperator{\aff}{aff}

\begin{document}

%
%
%
%
%
%
%
%
%

\title{Ball and Spindle Convexity with respect to a Convex Body}

\author{Zsolt L\'angi}
\address{Dept. of Geometry, Budapest University of Technology and Economics,\\
Egry J\'ozsef u. 1,\\
Budapest, Hungary, 1111}
\email{zlangi@math.bme.hu}

\author{M\'arton Nasz\'odi}
\address{Dept. of Geometry, E\"otv\"os University,\\
P\'azm\'any P\'eter s. 1/c,\\
Budapest, Hungary, 1117}
\email{nmarci@math.elte.hu}

\author{Istv\'an Talata}
\address{Dept. of Geometry and Informatics, Ybl Faculty of Architecture and Civil Engineering, Szent Istv\'an University,\\
Th\"ok\"oly \'ut 74,\\
Budapest, Hungary, 1146}
\email{talata.istvan@ybl.szie.hu}

\thanks{The first named author was supported by the J\'anos Bolyai Research Scholarship of the Hungarian Academy of Sciences.\\
The second named author was partially supported by the Hung. Nat. Sci. Found. (OTKA) grants: K72537 and PD104744.\\
The third named author was partially supported by the Hung. Nat. Sci. Found. (OTKA) grant no. K68398.}

\subjclass{52A30; 52A35; 52C17}

\keywords{ball convexity, spindle convexity, ball-polyhedron, separation, Carath\'eodory's theorem,
convexity structure, illumination, arc-distance}


\begin{abstract}
Let $C\subset\mathbb{R}^n$ be a convex body. 
We introduce two notions of convexity associated to C.
A set $K$ is $C$-\emph{ball convex} if it is the intersection of  translates of $C$, or it is either $\emptyset$, or $\mathbb{R}^n$.
The $C$-ball convex hull of two points is called a $C$-spindle.
$K$ is $C$-\emph{spindle convex} if it contains the $C$-spindle of any pair of its points.
We investigate how some fundamental properties of conventional convex 
sets can be adapted to $C$-spindle convex and $C$-ball convex sets.
We study separation properties and Carath\'eodory numbers of these
two convexity structures.
We investigate the basic properties of arc-distance, a quantity 
defined by a centrally symmetric planar disc $C$, which is the length 
of an arc of a translate of $C$, measured in the $C$-norm, that 
connects two points.
Then we characterize those $n$-dimensional convex bodies $C$ for which every $C$-ball 
convex set is the $C$-ball convex hull of finitely many points.
Finally, we obtain a stability result concerning covering numbers of some $C$-ball convex sets, 
and diametrically maximal sets in $n$-dimensional Minkowski spaces.
\end{abstract}

\maketitle

\section{Introduction}\label{sec:intro}

Closed convex sets may be introduced in two distinct manners: either 
as intersections of half-spaces, or as closed sets which contain the 
line segments connecting any pair of their points.
We develop these approaches further to obtain the notions of \emph{ball convexity} and
\emph{spindle convexity} with respect to a convex body. Let $C$ be a convex body (a compact convex set with non-empty interior)
in Euclidean $n$-space, $\Ren$.

\begin{defn}
A set $K \subset \Ren$ is called \emph{ball convex} with respect to $C$ (shortly, \emph{$C$-ball convex}),
if it is either $\emptyset$, or $\Re^n$, or the intersection of all translates of $C$ that contain $K$.
\end{defn}

\begin{defn}\label{defn:C-spindle}
Consider two (not necessarily distinct) points $p, q \in \Re^n$ such that there is a translate of $C$ that contains both $p$ and $q$.
Then the \emph{$C$-spindle} (denoted as $[p,q]_C$) of $p$ and $q$ is the intersection of all translates of $C$
that contain $p$ and $q$. If no translate of $C$ contains $p$ and $q$, we set $[p,q]_C = \Re^n$.
We call a set $K \subset \Re^n$ \emph{spindle convex} with respect to $C$ (shortly, \emph{$C$-spindle convex}),
if for any $p,q \in K$, we have $[p,q]_C \subset K$.
\end{defn}

In this paper, we study the geometric properties of ball and spindle convex sets depending on $C$,
and how these two families are related to each other.

In 1935, Mayer \cite{M35} defined the notion of ``\"Uberkonvexit\"at''. 
His definition coincides with our spindle convexity with the additional assumption that $C$ is smooth and strictly convex.
Attributing the concept to \cite{Mei11}, Valentine (p.99 of \cite{Val64}) uses the term ``$r$-convex'' for ball convex 
sets in a Minkowski space (where the radius of the intersecting balls is $r$). 
For other generalized notions of convexity related to Minkowski spaces, see \cite{MaSw2}.

Intersections of (finitely many) Euclidean unit balls on the plane
were studied by Bieberbach \cite{B55}, and in three dimensions by Heppes \cite{He56},
Heppes and R\'ev\'esz \cite{HR56}, Straszewitcz \cite{Str57} and Gr\"unbaum \cite{Gr57}.
Recent developments have led to further investigations of sets that are ball convex
with respect to the Euclidean unit ball $B^n$.
For these results, the reader is referred to \cite{BCCs06}, \cite{BN06}, \cite{B70}, \cite{KMP10}, \cite{KMW96}, \cite{KMP05} and \cite{AP95}.
 A systematic investigation of these sets was started by Bezdek et al. \cite{BLNP07}:
the authors examined how fundamental properties of convex sets can
be transferred to sets that are ball convex with respect to $B^n$;
in particular, they gave analogues of Kirchberger's and  Carath\'eodory's theorems, examined separation
properties of ball convex sets and variants of the Kneser-Poulsen conjecture.
It is also shown there that the notions of ball and spindle convexity coincide when $C=B^n$.
In the spirit of the results in \cite{BLNP07}, in \cite{LN08} generalizations of 
Kirchberger's theorem are proved regarding separation of finite point sets by homothets or 
similar copies of a given convex region.
Using ball-polyhedra, the authors of \cite{KMP10} give a
characterization of finite sets in Euclidean 3-space 
for which the diameter of the set is attained a maximal number of times.
The notion of $C$-ball convex hull (cf. Definition~\ref{defn:spindleconvexhull})
was defined for centrally symmetric plane convex bodies by Martini and Spirova \cite{MS09}.

In Section \ref{sec:notation}, we introduce our notation and basic concepts.
In Section~\ref{sec:separation}, we examine how the notions of ball 
and spindle convexity are related to each other and how sets are separated by translates of $C$.
In Section~\ref{sec:arcdistance}, we define the arc-distance between
two points with respect to a planar disc $C$, and examine
when the triangle inequality holds and when it fails.
In Section~\ref{sec:caratheodory}, we introduce and examine the 
Carath\'eodory numbers associated to these convexity notions.
In Section~\ref{sec:ballspindle}, we give a partial characterization of convex bodies $C$ for which 
every $C$-ball convex set is the $C$-ball convex hull of finitely many points.
Finally, in Section~\ref{sec:covering}, we prove that the operation of taking
intersections of translates of $C$ is stable in a certain sense. By applying
this result to Hadwiger's Covering Problem for certain $C$-ball convex
sets, and diametrically maximal sets in a Minkowski space, we obtain
a stability of upper bounds on covering numbers.

\section{Spindle and ball convex hull, convexity structures}\label{sec:notation}

We use the standard notation $\bd S$, $\inter S$, $\relint S$, $\aff S$, $\conv (S)$ and $\card S$ for
the \emph{boundary}, the \emph{interior}, the \emph{relative interior}, the \emph{affine hull}, the (linear) \emph{convex hull} and the \emph{cardinality}
of a set $S$ in $\Ren$. For two points $p,q \in \Re^n$, $[p,q]$ denotes the closed segment connecting $p$ and $q$.
The vectors $e_1, e_2, \ldots, e_n \in \Re^n$ denote the standard orthonormal basis of the space,
and for a point $x \in \Ren$, the coordinates with respect to this basis are $x=(x_1,x_2, \ldots, x_n)$.
The Euclidean norm of $p\in \Re^n$ is denoted by $|p|$. 
As usual, $\alpha A+\beta B$ denotes the Minkowski combination of sets 
$A,B\subset\Ren$ with coefficients $\alpha, \beta \in \Re$ (cf. \cite{Sch93}).
By an \emph{$n$-polytope} we mean an $n$-dimensional convex body
which is the (linear) convex hull of finitely many points.

\begin{defn}\label{defn:BB}
Let $C\subset\Ren$ be a convex body, $X \subseteq \Re^n$ a nonempty set and $r\geq0$. Then we set
\[
\Bp(X,r) = \bigcap_{v \in X} (rC+v) \quad \hbox{and} \quad \Bm(X,r) = \bigcap_{v \in X} (-rC+v)
\]
Furthermore, we set $\Bp(\emptyset,r) = \Bm(\emptyset,r) = \Re^n$. When $r$ is omitted, it is one:
$\Bp(X)=\Bp(X,1), \Bm(X)=\Bm(X,1)$. When $C=-C$ we may omit the $+/-$ signs.
\end{defn}

Note that by Definition~\ref{defn:C-spindle}, we have
\[
[p,q]_C = \Bp \Bm \left(\{ p,q\} \right).
\]

In the paper we use the following two fundamental concepts.

\begin{defn}\label{defn:spindleconvexhull}
The \emph{spindle convex hull of a set $A$ with respect to $C$} (in short, \emph{$C$-spindle convex hull}),
denoted by $\sconv{A}$, is the intersection of all sets that contain $A$ and are spindle convex with respect to $C$.
The \emph{ball convex hull of $A$ with respect to $C$} (or \emph{$C$-ball convex hull}), denoted by $\bconv{A}$,
is the intersection of all $C$-ball convex sets that contain $A$.
\end{defn}

We remark that $\bconv{A}$ is the intersection of those translates of $C$ that contain $C$, or in other words, $\bconv{A} = \Bp\Bm (A)$.

Next, we study the notions of ball and spindle convexity in the context of the theory of abstract convexity.

\begin{defn}\label{defn:convstr}
A \emph{convexity space} is a set $X$ together with a collection
$\GG \subseteq {\mathcal P}(X)$ of subsets of $X$ that satisfy

\renewcommand{\theenumi}{\roman{enumi}}
\begin{enumerate}
\item $\emptyset, X \in\GG$, and
\item $\GG$ is closed under arbitrary intersection.
\end{enumerate}
\end{defn}

Such a collection $\GG$ of subsets of $X$ is called a \emph{convexity structure} on $X$. If a third condition

(iii) $\GG$ is closed under the union of any increasing chains (with respect to inclusion)

also holds, then we call the pair $(X,\GG)$ an \emph{aligned space} (and $\GG$ an \emph{aligned space structure}). 
This is the terminology used, for example, by Sierksma \cite{Sie1984} and by Kay and Womble \cite{KW71}.

We note that in the literature (cf. van de Vel \cite{vdV93}),
if $\GG$ satisfies (i) and (ii), but does not necessarily satisfy 
(iii), then it is often referred to as a Moore family, or a closure 
system. On the other hand, `convexity space' frequently 
refers to what we call an aligned space (see \cite{K02}). Other 
terms used for an aligned space in the literature are 'domain finite 
convexity space' and 'algebraic closure system'.

The \emph{convex hull} operation associated with a convexity space $(X,\GG)$ is: 
$\conv_\GG(A)=\cap\{G\in\GG\ : A\subseteq G\}$ for any $A\subseteq X$.
The roughest convexity (resp. aligned space) structure $\GG$ which contains
a given family $\mathcal S\subseteq\mathcal P(X)$ is the \emph{convexity (resp. aligned space) structure generated by $\mathcal S$}.
This is the intersection of all convexity (resp. aligned space) structures which contain $\mathcal S$.

For a convex body $C\subseteq\Ren$, we denote the family of $C$-ball convex sets by $\BB_C$.
Clearly, $(\Ren,\BB_C)$ is a convexity space.
We note that the family of $C$-spindle convex sets is an aligned space structure, while
the family of closed $C$-spindle convex sets is a convexity structure.
Furthermore, the space of $C$-spindle convex bodies is a geometrical aligned space (an aligned space $(X,\GG)$ is called geometrical 
if $A=\bigcup \{\mathrm{conv}_{\GG}(F)\mid F\subseteq A, \card(F)\leq 2\}$ for every $A\in \GG$, see \cite{K02}).

Clearly, $\BB_C$ is the convexity space generated by the translates of $C$.
Let $\TT_C$ denote the aligned space structure generated by these translates.
In Theorem~\ref{thm:convstruct}, we compare their corresponding hull operations $\bconvures$ and $\conv_{\TT_C}$.
For the proof we need the following result of Sierksma \cite{Sie1984}.

\begin{lem}[Theorem~7 in \cite{Sie1984}]\label{lem:Sie}
Let $(X,\mathcal S)$ be a convexity space, and let $\GG$
be the aligned space structure generated by $\mathcal S$. If $A\subseteq X$, then
\[
\conv_\GG (A)=\bigcup_{k=0}^\infty \bigg[ \cup \{ \conv_{\mathcal S}(F): F\subseteq A, \card(F)\leq k\} \bigg].\]
\end{lem}

\begin{thm}\label{thm:convstruct}
Let $C$ be a convex body in $\Ren$, and let $A\subseteq \Ren$. Assume that $\dim\bconv{A}=n$.
Then
\begin{equation}\label{eq:convstruct}
\cl\left(\conv_{\TT_C} (A)\right)=\bconv{A},
\end{equation}
that is, $\cl\left(\conv_{\TT_C} (A)\right)$ is the intersection of all translates of $C$ that contain $A$.
\end{thm}

\begin{proof}
Clearly, $\cl\left(\conv_{\TT_C} (A)\right)\subseteq\Bp\Bm(A)$.
To prove the reverse containment, assume that $x\in\inter\Bp\Bm(A)$.
By Lemma~\ref{lem:Sie}, it is sufficient to show that $x\in\Bp\Bm (F)$ for some finite $F\subseteq A$.

Since $-\inter C+x\supseteq \Bm (A)$ (see Remark~\ref{rem:bconvcheck}), we have 
$(-\bd C+x)\cap \Bm (A)=\emptyset$.
From the compactness of $-\bd C+x$, it follows that there is a
finite subset $F$ of $A$ with $(-\bd C+x)\cap \Bm (F)=\emptyset$.
Thus, $x\in \Bp\Bm (F)$.
\end{proof}

Clearly, if any $C$-spindle is $n$-dimensional, then so is the $C$-ball convex hull of any set containing more than one point.
This leads to the following observation.

\begin{rem}\label{rem:ndimensional}
If $C$ is a strictly convex body in $\Ren$, then $\cl (\conv_{\TT_C} (A)) = \bconvures (A)$ for any set $A \subset \Ren$.
\end{rem}
We expect a positive answer to the following question:
\begin{prob}\label{prob:convstruct}
Can we drop the condition on the dimension of $\bconv A$ in Theorem~\ref{thm:convstruct}?
\end{prob}

\section{Relationship between spindle and ball convexity, and separation by translates of a convex body}\label{sec:separation}

Clearly, for any convex body $C$, a $C$-ball convex set is closed and $C$-spindle convex.
Thus, for any $X$ and $C$ the $C$-ball convex hull of $X$ contains its $C$-spindle convex hull.
In \cite{BLNP07}, it is shown that if $C$ is the Euclidean unit ball, then for closed sets
the notions of spindle and ball convexity coincide.
Now we show that it is not so for every convex body $C$.

\begin{ex}\label{ex:convexhull}
We describe a $3$-dimensional convex body $C$ and a set $H\subseteq \Re^3$ for which 
$H$ is $C$-spindle convex but it is not $C$-ball convex.
Let $T \subset \Re^3$ be a regular triangle in the $x_3=0$ plane, with the origin as its centroid
(cf. Figure~\ref{fig:convexhull}).
Let $C = \conv \big((T+e_3) \cup (-T-e_3)\big)$.
Let $H$ be the intersection of $C$ with the plane with the equation $x_3=0$.
Note that $H$ is a regular hexagon: $H=(T-T)/2$.

\begin{figure}[ht]
\begin{center}
\includegraphics[width=0.45\textwidth]{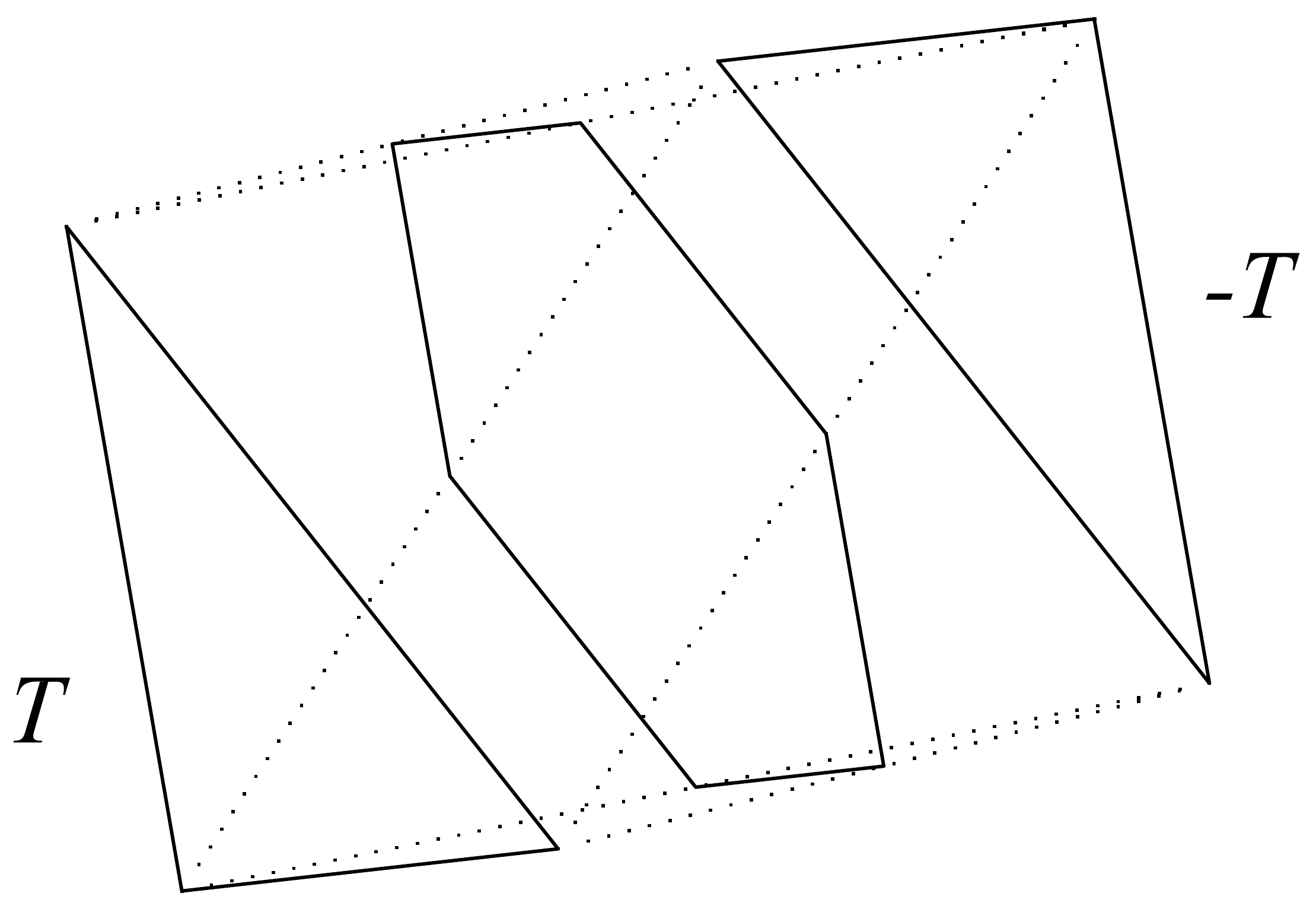}
\caption[]{Example~\ref{ex:convexhull}}
\label{fig:convexhull}
\end{center}
\end{figure}

We show that $H$ is $C$-spindle convex.
Note that $H$ and $T$ are of constant width two in the two-dimensional norm defined by $H$.
Thus, for any $p, q \in H$, there is a chord of $T$, parallel to $[p,q]$, that is not shorter than $[p,q]$.
From this, it follows that there is a translate $T+z_1$ of $T$ that contains $p, q$. Similarly, there is a set
$-T+z_2$ containing $p$ and $q$. Denote by $\mathcal K^n$ the collection of all $n$-dimensional convex bodies.
We observe that $[p,q]_C \subseteq C \cap (C+z_1-e_3) \cap (C+z_2+e_3)\subseteq H$, which yields the desired statement.

We have that $\sconv{H} = H$. However, clearly, there is only one translate of $C$ containing $H$, namely $C$.
Thus, $H=\sconv{H} \neq \bconv{H}=C$.
\end{ex}

We introduce the following notions. Note the order of $X$ and $Y$ in Definition~\ref{defn:separation}.

\begin{defn}\label{defn:separation}
Let $C \subset \Re^n$ be a convex body, and let $X, Y \subset \Re^n$.
We say that a translate $C+x$ of $C$ \emph{separates $X$ from $Y$}
if $X \subset C+x$, and $\inter (C+x) \cap Y = \emptyset$. If $X \subset \inter(C+x)$ and $(C+x) \cap Y = \emptyset$, then we say that $C+x$ \emph{strictly separates $X$ from $Y$}.
\end{defn}

In \cite{BLNP07}, it is proved that if $C=B^n$ is the Euclidean ball, then any $C$-spindle convex set
is separated from any non-overlapping $C$-spindle convex set by a translate of $C$.
By Example~\ref{ex:convexhull}, not all convex bodies have this property (there, $H$ is not separated from any singleton subset of $C$).
Thus, we introduce the following notions.

\begin{defn}\label{defn:separability}
Let $C \subset \Re^n$ be a convex body, and let $K \subset \Re^n$ be a $C$-spindle convex set.
We say that \emph{$K$ satisfies Property (p), (s) or (h) with respect to $C$}, if
\begin{itemize}
\item for every point $p \notin K$, there is a translate of $C$ that separates
$K$ from $p$ (Property (p)),
\item for every $C$-spindle convex set $K'$ that does not overlap with $K$ there is a translate of $C$ that separates
$K$ from $K'$ (Property (s)),
\item for every hyperplane $H$ with $H \cap \inter K = \emptyset$, there is a translate of $C$ that separates $K$ from $H$ (Property (h)).
\end{itemize}
\end{defn}

It is not difficult to show that Property (h) yields (s), which in turn yields (p).

\begin{rem}\label{rem:pointseparability}
A closed set $K$ satisfies (p) if, and only if, $K$ is $C$-ball convex.
In particular, for closed sets the notions of spindle convexity and ball convexity with respect to a convex body $C$ coincide if, and only if, every closed $C$-spindle convex set satisfies (p).
\end{rem}

\begin{rem}\label{rem:smoothsh}
For a smooth convex body $C$, (s) and (h) are equivalent.
\end{rem}

Recall that in Example~\ref{ex:convexhull}, $H$ is $C$-spindle convex but not $C$-ball convex.
We note that $C$ may be replaced by a smooth $C'$ such that $H$
is $C'$-spindle convex but not $C'$-ball convex. 
Simply, apply the following theorem for the convex body $C$ of Example~\ref{ex:convexhull} and any smooth and strictly convex body $D$.
Then it follows that $H$ of Example~\ref{ex:convexhull} is $C'$-spindle convex for $C'=C+D$, 
but it is easy to see that $H$ is not $C'$-ball convex.

\begin{thm}\label{thm:addition}
Let $C,D$ be convex bodies in $\Re^n$ and let $S\subseteq \Re^n$. If $S$ is $C$-ball convex, then $S$ is $(C+D)$-ball convex. Similarly, if $S$ is $C$-spindle convex, then $S$ is $(C+D)$-spindle convex.
\end{thm}

We need the following standard lemma, for a proof see Lemma~3.1.8. in \cite{Sch93}.

\begin{lem}\label{lem:addition}
If $C,D$ are convex bodies in $\Re^n$, then $C=\bigcap_{x\in D}(C+D-x)$.
\end{lem}

\begin{proof}[Proof of Theorem~\ref{thm:addition}]
Observe that by Lemma~\ref{lem:addition}
we have that $\conv^b_{C+D}(S)\subseteq \conv^b_C(S)$, and $\conv^s_{C+D}(S)\subseteq \conv^s_C(S)$, for any $S\subseteq \Re^n$.
These readily imply the statement of the theorem concerning ball convexity.
The statement concerning spindle convexity follows from the fact that for any two points $p,q\in S$ we have
$[p,q]_{C+D}=\bconv[C+D]{\{p,q\}}  \subseteq \bconv[C]{\{p,q\}}=[p,q]_{C}$.
\end{proof}

A frequently used special case of Theorem~\ref{thm:addition} is the following.

\begin{cor}\label{cor:addition}
Let $C$ be a convex body in $\Re^n$, let $S\subseteq \Re^n$, and let $0<r<1$ be arbitrary. If $S$ is $C$-ball convex, then $rS$ is $C$-ball convex. 
In particular, $rC$ is a $C$-ball convex set. Similarly, if $S$ is $C$-spindle convex, then $rS$ is $C$-spindle convex. In particular, $rC$ is $C$-spindle convex.
\end{cor}

\begin{proof}
We apply Theorem~\ref{thm:addition} to $rC$ and $(1-r)C$.
\end{proof}

For $n\geq 3$, the analogous implication of Theorem~\ref{thm:addition} is not true for $S+D$ in general, 
so ball convexity is not preserved in general by adding a convex body $D$ to both a $C$-ball convex set $S$ and to $C$.
The same holds for spindle convexity. We show both in the following example.

\begin{ex}\label{ex:addition}
We describe convex bodies $C, D \subset \Ren$ and a set $S\subset \Ren$, for any $n\geq 3$, such that $S$ is $C$-ball convex 
(and thus $S$ is $C$-spindle convex), and $S+D$ is not $(C+D)$-spindle convex. We note that all sets $C, D, S \subseteq \Ren$ will be centrally symmetric.

Let $C=\conv([-1,1]^n\cup\{\pm (1+\epsilon) e_i\}_{i=1}^{n-1})$, where $0<\epsilon<1$. Let $D=rB^n$, that is, 
$D$ is a Euclidean ball of radius $r$ for some $r>0$, centered at the origin. Choose $S=[e_n,-e_n]$. Then $S=\bconv{S}=\sconv{S}$ since $S=\left(C+\sum_{i=1}^{n-1}e_i\right)\cap \left(C-\sum_{i=1}^{n-1}e_i\right)$.

On the other hand, $(S+D)\cap \{x_n=0\}=rB^{n-1}\times \{0\}$, and thus $\pm re_i\in \bd(S+D)$ for any $1\leq i\leq n-1$, while we have
$$
\pm re_i\in \inter\bigg(\conv^s_{C+D}\Big((1+r)e_n\cup\big(-(1+r)e_n\big)\Big)\bigg)\subseteq \inter\big(\conv^s_{C+D}(S+D)\big),
$$
for any $1\leq i\leq n-1$, so $(S+D)\subsetneq \conv^s_{C+D}(S+D)$.
\end{ex}

In the following, we show that Example~\ref{ex:convexhull} is not a `rare phenomenon'.

\begin{thm}
Let $n \geq 3$. The family of those smooth $n$-dimensional convex bodies for which the associated ball and spindle convexity do not coincide, forms an everywhere dense set in the metric space of the $n$-dimensional convex bodies, equipped with the Hausdorff metric.
\end{thm}

\begin{proof}
Let $C$ be any convex body. Note that there are two distinct points $x,y$ of $C$ with a hyperplane $F$ through the origin such that the parallel hyperplanes $F+x$ and $F+y$ support $C$ and their intersection with $C$ is $\{ x \}$ and $\{ y \}$, respectively.

Let $T$ be a small $(n-1)$-dimensional regular simplex in $F$ with the origin as its centroid.
Let $C' = \conv (C \cup (T+x) \cup (-T+y))$ and let $H=(T-T)/2$. Similarly to Example~\ref{ex:convexhull},
we observe that $H$ is an $(n-1)$-dimensional $C'$-spindle convex set, whereas $\bconv[C']{H}$ is $n$-dimensional.

To construct a smooth body with respect to which ball and spindle convexity do not coincide, we take $H$ and $\bar C=C'+\rho B^n$. If $\rho>0$ is sufficiently small, then $\bar C$ is close to $C$.
\end{proof}

Now we present an example of $C$ and a closed convex set $H$ where $H$ satisfies Property (p), i.e., it is $C$-ball convex, but does not satisfy (h).

\begin{figure}[ht]
\begin{center}
\includegraphics[width=0.47\textwidth]{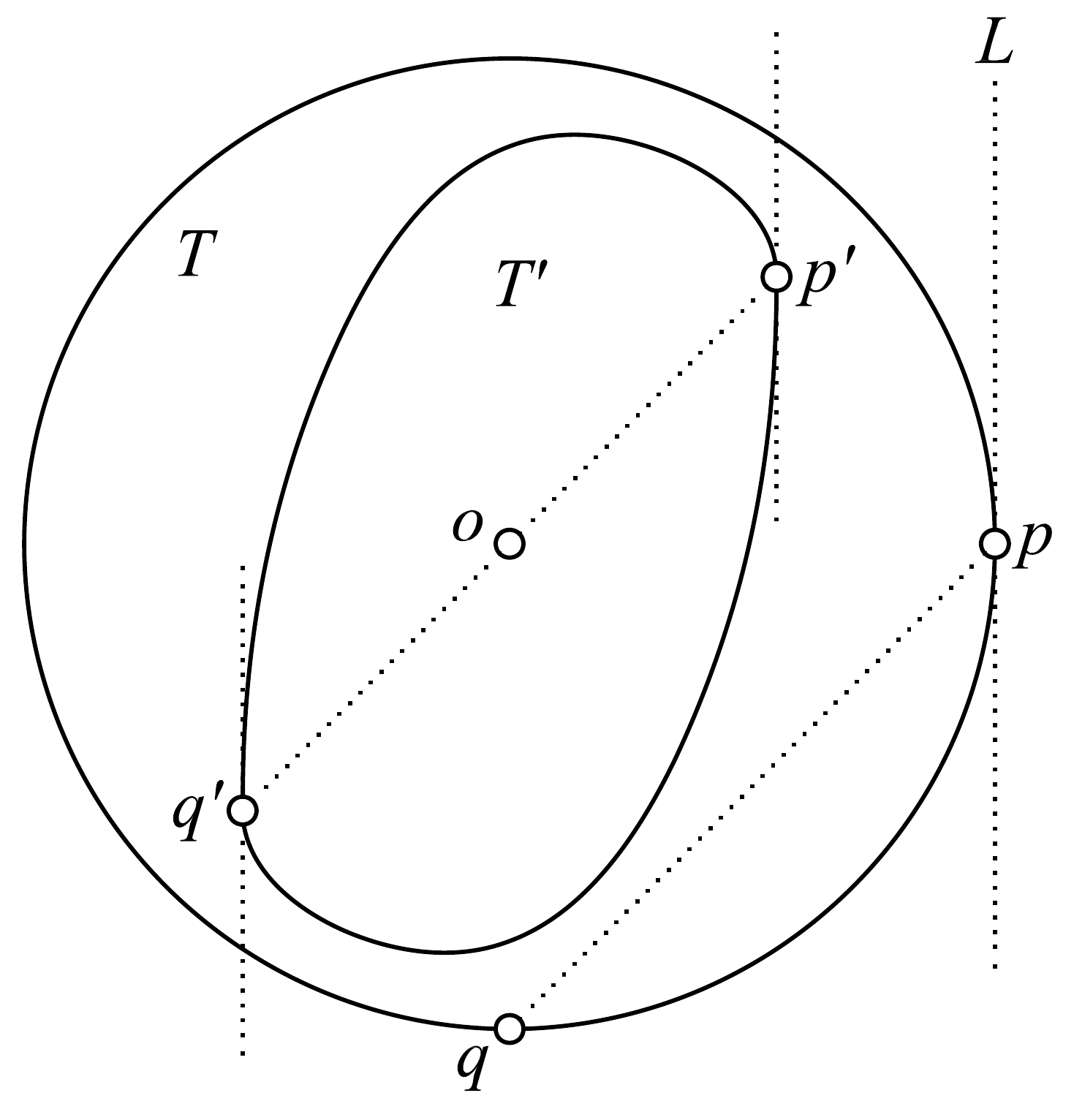}
\caption[]{Example~\ref{ex:convexseparation}}
\label{fig:sepdifferent}
\end{center}
\end{figure}

\begin{ex}\label{ex:convexseparation}
Let $T$ be the Euclidean unit disk centered at the origin in the plane $U=\{x_3=0 \}$ in $\Re^3$, and consider the points $p=e_1$ and $q=-e_2$.
Let $T' \subset T+e_3$ be a smooth plane convex body in the $U+e_3$ plane with the following properties: 
$T'$ is symmetric about $e_3$, and the chord $[p',q']$ of $T'$ connecting its supporting lines
parallel to the $x_2$-axis is parallel to and slightly longer than $[p,q]$.
We choose our notation in a way that $T'+p-p'$ and $T$ are on the same side of the line $L$ defined by the equations $x_1=1, x_3=0$.
Let $p^*$ be the reflection of $p'$ in the plane $U$,
and set $C=\conv (T \cup T' \cup (-T'))$.
Figure~\ref{fig:sepdifferent} shows the $x_3=0$ and $x_3=1$ sections of $C$, viewing it from the top.
Let $H=(C+p-p') \cap (C+p-p^*)$  and observe that $H= T'+p-p'$ is a $C$-ball convex set contained in $U$.
Consider the plane with the equation $x_1=1$, and note that it supports $H$ at $p$.

If a translate $C+v$ of $C$ separates $H$ from the plane $x_1=1$ then $p\in C+v$, and $x_1=1$ is a support plane of $C+v$. It follows that $v=0$.
However, $C \nsupseteq H$ because $|p^\prime-q^\prime|>|p-q|$.
\end{ex}

A suitable modification of this example yields a smooth convex body $C$ arbitrarily close
to the Euclidean unit ball such that some $C$-ball convex sets (which thus satisfy (p)) do not satisfy (h).
Hence, by Remark~\ref{rem:smoothsh}, there are spindle convex sets that satisfy (p) and do not satisfy (s) for some
convex body $C$.

We propose the following problems.

\begin{prob}
Is there a spindle convex set $K$ that satisfies (h) and does not satisfy (s) with
respect to some convex body $C$?
\end{prob}

\begin{prob}
Is there a convex body $C$ such that every $C$-spindle convex set $K$ satisfies
(p) (respectively (s)), but at least one of them does not satisfy (s) (respectively (h))?
\end{prob}

We conclude this section by finding some special classes of convex bodies
such that any spindle convex set with respect to any of  them satisfies Property (h).
Thus, if a set is closed and  spindle convex with respect to any of them, then it is ball convex with respect to the same convex body.
First, we recall a standard definition.

\begin{defn}\label{defn:Cdistance}
The central symmetral $\frac{1}{2}(C-C)$ of a convex body $C \subset \Re^n$  defines
a norm on $\Re^n$, called the \emph{relative norm of $C$} (cf. \cite{E65}).
We recall that in this norm, $C$ is a convex body of constant width two.
For points $p, q \in \Re^n$, we call the distance between $p$ and $q$, measured in the norm relative to $C$,
the \emph{$C$-distance of $p$ and $q$} (cf. Lassak \cite{L91}).
For a set $X$, we denote the diameter of $X$ measured in $C$-distance by $\diam_C X$.
If $C$ is the Euclidean ball, we may write simply $\diam X$.
\end{defn}

\begin{thm}\label{thm:separation2d}
Let $C \subset \Re^2$ be a plane convex body. Then any $K \subset \Re^n$ $C$-spindle convex set satisfies (h).
\end{thm}

\begin{proof}
Clearly, we may assume that $K$ is closed.
We show the assertion for the case that $K$ is a plane convex body.
We leave it to the reader to verify it in the cases that $K$ is not bounded or has an empty interior.

Let $K$ be a $C$-spindle convex body, and $L$ be any line supporting $K$.
Since translating $C$ does not change whether $K$ satisfies (h) or not, we may assume that $L$ supports also $C$.
Let $L'$ be the other supporting line of $C$ parallel to $L$.

First, consider the case that $C \cap L$ is a singleton $\{ x \}$. Note that as $K$ is $C$-spindle
convex, we have that $K \cap L$ is also a singleton, since otherwise the $C$-spindle of the endpoints of $K \cap L$
has a point on the side of $L$ not containing $K$.
Without loss of generality, we may assume that $K \cap L = \{ x \}$.

Assume that $K \not\subseteq C$, and consider a point $y \in K \setminus C$.
Observe that there is a translate of $C$ containing $x$ and $y$ if, and only if their $C$-distance is at most two.
Thus, $y$ is contained in the closed unbounded strip between $L$ and $L'$.
Note that if we sweep through $C$ by a family of parallel lines, the $C$-length of the intersecting segment strictly increases
while its length reaches two, then it may stay two for a while,
then it strictly decreases until it reaches the other supporting line of $C$.
Consider the chords of $C$ that are parallel to and not shorter than $[x,y]$, and observe that, since $y \notin C$, they all are
on the side of the line, passing through $x$ and $y$, that contains $L' \cap C$ (cf. Figure~\ref{fig:planeseparationsinglepoint}).
From this, it follows that the intersection of all the translates of $C$ containing $x$ and $y$
(or, in other words, the $C$-spindle of $x$ and $y$) has a point on the side of $L$ not containing $C$.
Since $K$ is $C$-spindle convex and $L$ supports $K$ at $x$, we arrived at a contradiction.

\begin{figure}[ht]
\begin{minipage}[b]{0.5\columnwidth}%
    \centering
    \includegraphics[width=0.9\textwidth]{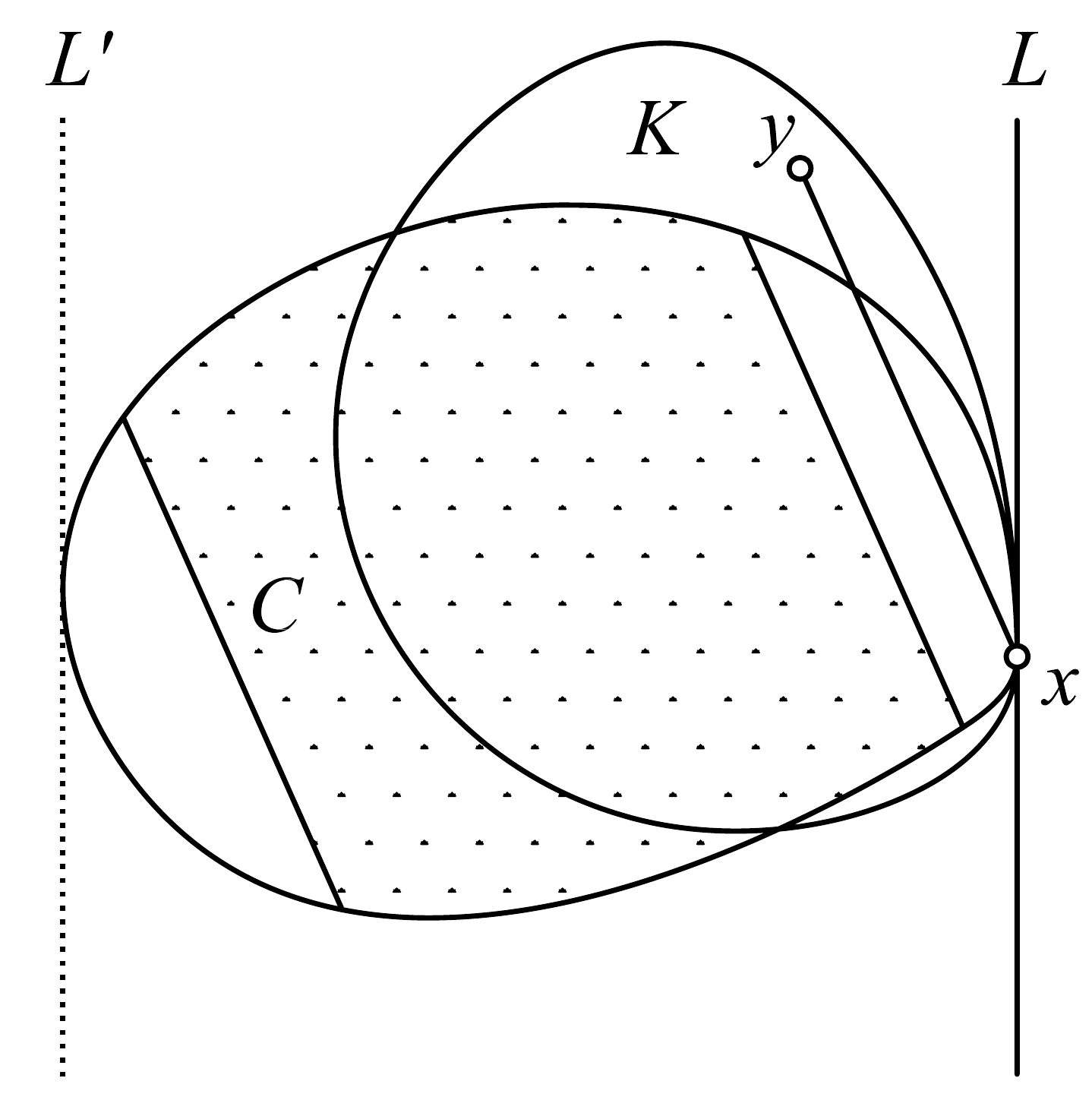}
    \caption[]{\\When $C\cap L=\{x\}$}
    \label{fig:planeseparationsinglepoint}
\end{minipage}%
\hfill%
\begin{minipage}[b]{0.5\columnwidth}%
\centering
    \includegraphics[width=0.9\textwidth]{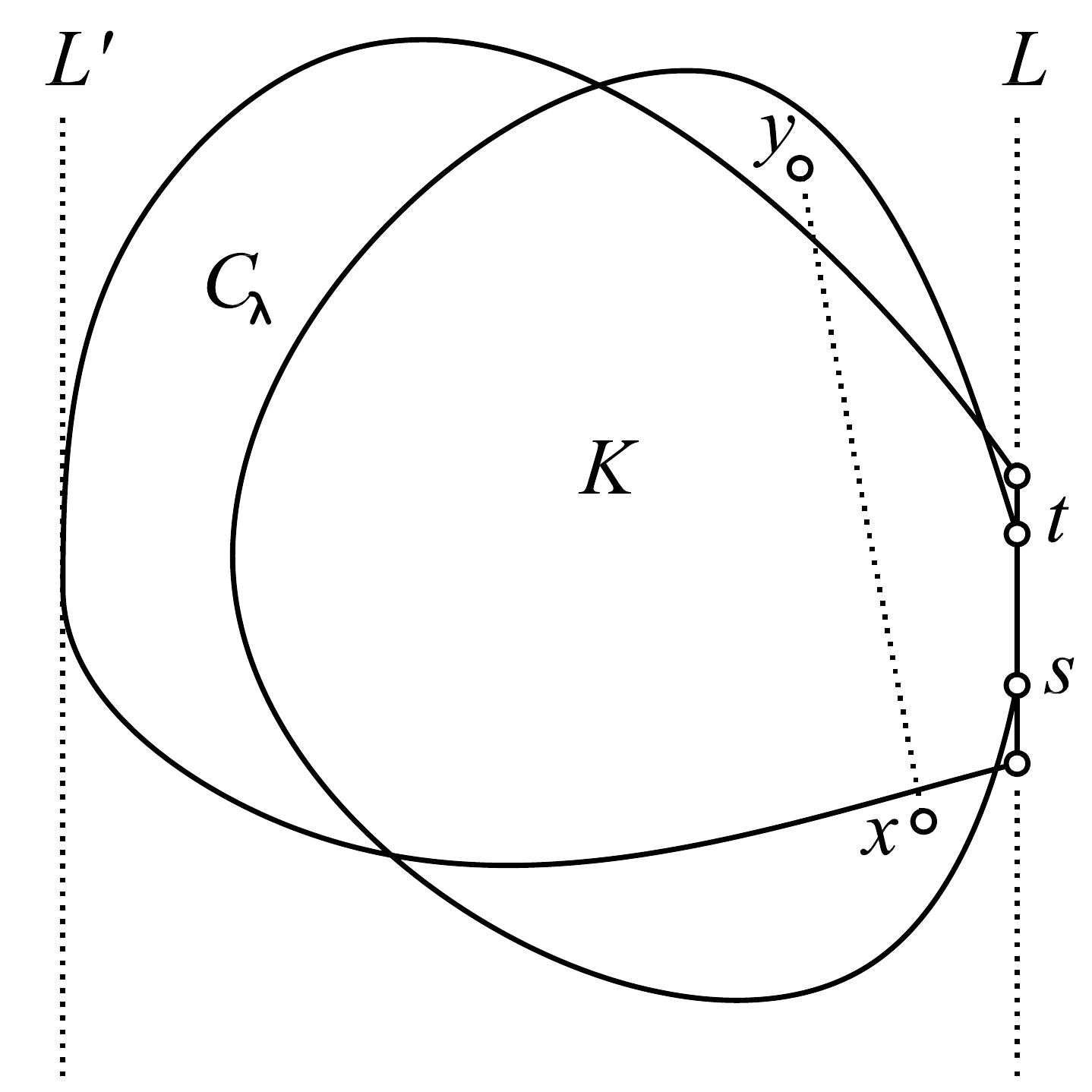}
    \caption[]{\\When $C\cap L=[p,q]$}
    \label{fig:planeseparation}
\end{minipage}
\end{figure}

Now consider the case that $C \cap L = [p,q]$ with $p \neq q$.
Let $K \cap L = [s,t]$.
Then $|s-t| \leq |p-q|$, and we may assume that $[s,t] \subset [p,q]$ and that $t \in [s,q]$.
If there is a point $v \in K$ not contained in the closed strip bounded by $L$ and $L'$, then the $C$-distance
of $v$ and $s$ is greater than two; a contradiction.
Set $C_0 = t-q+C$, $C_1 = s-p+C$ and $u=q-p-(s-t)$, which yields that $C_1 = u+C_0$.

Assume that there is no translate of $C$, with supporting line $L$, that contains $K$.
Then, $C_\lambda \setminus K \neq \emptyset$ for every $\lambda \in [0,1]$, where $C_\lambda = \lambda u +C_0$.
Let $H$ denote the closed strip bounded by $L$ and $L'$.
For simplicity, we regard the connected component of $H \setminus C_\lambda$, not containing $q$ as the one
\emph{below} $C_\lambda$, and the other one as the one \emph{above} $C_\lambda$.
It is easy to see that if there is no $\lambda$ with $K \subset C_\lambda$, then
there is a value of $\lambda$ such that both components of $H \setminus C_\lambda$ contain
a point of $K$.
Let $\lambda$ be such a value, and let $x,y \in K$ be points such that $x$ is in the component
below $C_\lambda$ and $y$ is in the component above $C_\lambda$ (cf. Figure~\ref{fig:planeseparation}).

Consider the segment $[x,y]$ and note that if their $C$-distance is greater than two,
then $[x,y]_C = \Re^2$.
Thus, all the chords of $C_\lambda$ that are parallel to and not shorter than $[x,y]$ are on the same side of the line containing $[x,y]$.
Similarly like in the previous case, it follows that $[x,y]_C \not\subset H$, which contradicts our assumption that $K \not\subseteq H$.
\end{proof}

\begin{cor}
For a plane convex body $C\subset\Re^2$, any closed $C$-spindle convex set is $C$-ball convex.
\end{cor}

\begin{prop}
Let $C_1 \subset \Re^k$ and $C_2 \subset \Re^m$ be convex bodies, and consider a set $S \subset \Re^k\oplus\Re^m=\Re^{k+m}$.
Let $\proj_1$ and $\proj_2$ be the orthogonal projections of $\Re^k\oplus\Re^m$ on the first and second factor, respectively.
Then,
\[
\bconv[C_1 \times C_2]{A}  =
\bconv[C_1]{\proj_1 A} \times \bconv[C_2]{\proj_2 A};
\]
\[
\sconv[C_1 \times C_2]{A} = \sconv[C_1]{\proj_1 A} \times \sconv[C_2]{\proj_2 A}.
\]
\end{prop}

\begin{proof}
Note that for any $x \in \Re^{k+m}$, $x + (C_1 \times C_2) = (\proj_1 x + C_1) \times (\proj_2 x + C_2)$.
Thus, $A \subseteq x+ (C_1 \times C_2)$ is equivalent to $\proj_1 A \subseteq (\proj_1 x + C_1)$
and $\proj_2 A \subseteq (\proj_2 x + C_2)$. This immediately yields the equality of the ball convex hulls.
It follows that for any points $p,q \in \Re^{k+m}$, we have $[p,q]_{C_1 \times C_2} = [\proj_1 p, \proj_1 q]_{C_1}
\times [\proj_2 p, \proj_2 q]_{C_2}$.
Hence, the right-hand side in the second equality contains the left-hand side.
Now consider any $(C_1 \times C_2)$-spindle convex set $K$.
We show that $\proj_1 K \times \proj_2 K \subseteq K$.
Consider points $p_1 \in \proj_1 K$ and $q_2 \in \proj_2 K$.
We need to show that $(p_1,q_2)\in K$.
Note that $(p_1,p_2) \in K$ and $(q_1,q_2) \in K$ for some $p_2 \in \Re^m$ and $q_1 \in \Re^k$.
Thus,  by applying the first equality for $(C_1 \times C_2)$-spindles, we obtain that $(p_1,q_2) \in K$, which is what we wanted to prove.
\end{proof}

\begin{cor}\label{cor:cubeconvexhull}
If $C$ is an $n$-dimensional axis-parallel cube, and $A \subseteq \Ren$ is any set, then both
$\bconv{A}$ and $\sconv{A}$ are either the axis-parallel box containing $A$ and minimal with respect to inclusion, or $\Ren$.
\end{cor}

\section{Arc-distance defined by $C$}\label{sec:arcdistance}

In the theory of spindle convexity with respect to the Euclidean disk, there is a naturally arising associated distance function. This distance is called \emph{arc-distance}, and is defined for points $p,q \in \Re^2$ as the Euclidean length of a shortest unit circle arc connecting the points (cf. \cite{BCCs06} and \cite{BLNP07}). The aim of this section is to generalize this distance
for spindle convexity with respect to any origin-symmetric plane convex body.

Let $C$ be a planar $o$-symmetric convex body, that is the unit disk of a normed plane.
We recall that the \emph{$C$-length of a polygonal curve} is the sum of the $C$-distances between the consecutive pairs of points, and that the \emph{$C$-length of a curve} is the supremum, if it exists, of the $C$-lengths of the polygonal curves for which all the vertices are chosen from the curve.
If $D$ is a convex body, then the \emph{perimeter of $D$ with respect to $C$} is the $C$-length of the boundary of $D$.
We denote this quantity by $\perim_C D$.
It is known that for any plane convex bodies $C$ and $D$, we have $\perim_C D = \perim_C (\frac{1}{2}(D-D))$ (cf. \cite{FM82}), and $6 \leq \perim_C C \leq 8$ (cf.
\cite{G32} or \cite{S67}).

\begin{defn}
Let $C$ be an $o$-symmetric plane convex body, and let $p,q$ be points at $C$-distance at most two.
Then the \emph{arc-distance $\arclength_C (p,q)$ of $p,q$ with respect to $C$} is the minimum of the $C$-lengths
of the arcs, with endpoints $p$ and $q$, that are contained in $\bd(y+C)$ for some $y \in \Re^2$.
\end{defn}

\begin{defn}
Let $C$ be an $o$-symmetric plane convex body, $z\in \Re^2$ and $0 \leq \rho \leq \frac{1}{2}\perim_C C$. Then
the \emph{arc-distance disk, with respect to $C$, of center $z$ and radius $\rho$} is the set
\[
\Delta_C(z,\rho) = \{ w \in \Re^2 : \arclength_C(z,w) \leq \rho \}.
\]
Furthermore, we set $\Delta_C(\rho)=\Delta_C(o,\rho)$.
\end{defn}

Clearly, for any $C$, arc-distance disks of the same radius are translates of each other,
but those of different radii are not necessarily even similar.
We note that if $C$ is a Euclidean disk, then its arc-distance disks are Euclidean disks.

\begin{thm}\label{thm:arclength}
For any $o$-symmetric plane convex body $C$ and $0 \leq \rho \leq \frac{1}{2} \perim_C C$, the arc-distance disk
$\Delta_C(\rho)$ is convex.
\end{thm}

\begin{proof}
Note that by compactness arguments, it is sufficient to prove the assertion for the case that $C$ is a convex polygon
with, say, $m$ vertices and for values of $\rho$ such that $0 \leq \rho \leq \frac{1}{2}\perim_C C$  is not equal to the sum of some sides of $C$.
In this case no chord of $C$ connecting two vertices determines an arc of length $\rho$.

Let us move a point $p(t)$ around on $\bd C$ at a constant speed measured in $C$-distance, and consider
the point $q(t)$ such that $\arclength_C(p(t),q(t)) =\rho$.
Note that on any side of $C$, the points $p(t)$ and $q(t)$ move at a constant speed also in the Euclidean metric,
and their Euclidean speed is proportional to the lengths of the longest chords of $C$ parallel to the corresponding
edges of $C$.
Thus, the vector $q(t)-p(t)$ is a linear function of $t$ if $p(t)$ and $q(t)$ are on different edges, and a constant
if they are on the same edge, which yields that $Q=\bd \Delta_C(\rho)$ is a (starlike) polygonal curve with at most $2m$ vertices.

\begin{figure}[ht]
\begin{center}
\includegraphics[width=0.45\textwidth]{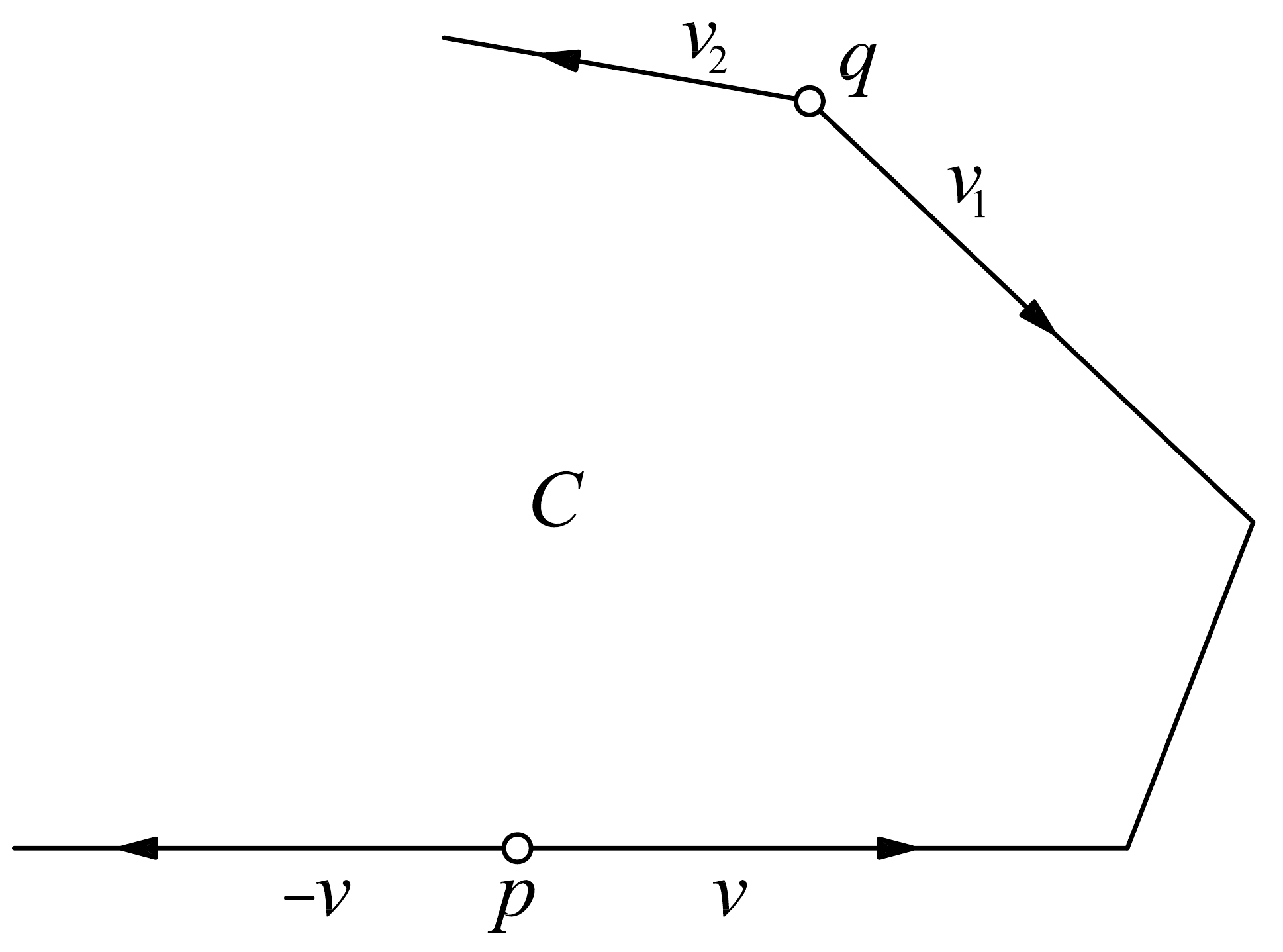}
\caption[]{Polygon $C$ in the proof of Theorem~\ref{thm:arclength}}
\label{fig:arclength}
\end{center}
\end{figure}

We show that $Q$ is convex.
Consider four points: a vertex of $Q$, one point on each edge meeting at that vertex and the origin.
Since $Q$ is clearly starlike with respect to the origin, it is sufficient to prove that 
for any choice of the vertex (and the points on the edges), these four points are in convex position.

Any vertex of $Q$ is of the form $q-p$, where $q$ is a vertex of $C$, and $p$ is a point in the relative interior of some edge of $C$. 

Consider the sufficiently small vectors $v$, $v_1$ and $v_2$ whose Euclidean lengths are proportional to the Euclidean lengths of the longest chords
of $C$ parallel to them, and such that $p-v$, $p$, $p+v$ are on the same edge of $C$,
and $q+v_1$ and $q+v_2$ are on the two consecutive edges of $C$ meeting at $q$.
We choose our notation in a way that the line passing through $[p,q]$ separates $p-v$ and $q+v_2$
from $p+v$ and $q+v_1$, and that $v$, $-v_1$ and $v_2$ are in counterclockwise order
(cf. Figure~\ref{fig:arclength}).
Note that as $C$ is convex, the segment $[q+v_1,q+v_2]$ intersects $[p,q]$.

Now, to show that $Q$ is convex, we need to show
that the points $o$,  $q+v_1-(p-v)$, $q-p$ and $q+v_2-(p+v)$ are in convex position.
Adding $p$ to each point, we are considering the quadrilateral with vertices
$p$, $q+v_1+v$, $q$ and $q+v_2-v$. It is clearly convex, by the convexity of $C$.
\end{proof}

We note that if $C$ is not $o$-symmetric, we may define a non-symmetric arc-distance, for which we may prove the assertion of Theorem~\ref{thm:arclength} using a similar technique.

\begin{rem}
If $C$ is smooth, then $\Delta_C(\rho)$ is smooth for any positive value of $\rho$.
\end{rem}

\begin{proof}
Consider the $C$-arc-length parametrization $\Gamma : [0,\alpha] \to \Re^2$ of $\bd C$.
Then $\bd \Delta_C(\rho)$ is the graph of the curve $\tau \mapsto \Gamma(\tau+\rho)- \Gamma(\tau)$.
If $\Gamma$ is differentiable, then so is this function.
\end{proof}

Our next corollary follows from Theorem~\ref{thm:arclength} and the fact that
$\Delta_C(\rho_1) \subset \Delta_C(\rho_2)$ for any $\rho_1 < \rho_2$.

\begin{cor}\label{cor:quasiconvex}
For any $x,y,z \in \Re^2$ and an $o$-symmetric plane convex body $C$, the function $\tau \mapsto \arclength_C(x,y+\tau z)$ is quasiconvex on its domain; that is:
it consists of a strictly decreasing, possibly a constant and then a strictly increasing interval.
\end{cor}

We prove the following version of the triangle inequality for arc-distance, which, for the Euclidean case, appeared first 
as Lemma~1 in \cite{BCCs06}.

\begin{thm}\label{thm:triangle}
Let $C$ be an $o$-symmetric plane convex body, and let $x,y,z \in \Re^2$ be points such that each pair has a $C$-arc-distance.
\begin{enumerate}
\item[(1)] If $y \in \inter [x,z]_C$, then $\arclength_C(x,y)+\arclength_C(y,z) \leq \arclength_C(x,z)$,
\item[(2)] if $y \in \bd [x,z]_C$, then $\arclength_C(x,y)+\arclength_C(y,z) = \arclength_C(x,z)$, and
\item[(3)] if $y \notin [x,z]_C$ and $C$ is smooth, then $\arclength_C(x,y)+\arclength_C(y,z) \geq \arclength_C(x,z)$.
\end{enumerate}
\end{thm}

\begin{proof}
The assertion of (2) is trivial. Assume that $y \in \inter [x,z]_C$. Clearly, there is a line $L$ 
containing $y$ such that $L \cap \bd [x,z]_C = \{ a,b \}$ with $\arclength_C(x,a)=\arclength_C(x,b)$ and $\arclength_C(z,a)=\arclength_C(z,b)$.
Thus, the assertion in this case is a consequence of Corollary~\ref{cor:quasiconvex}.

Assume that $C$ is smooth and $y \notin [x,z]_C$.
If $\arclength_C(x,y) \geq \arclength_C(x,z)$, then there is nothing to prove,
and thus, we may assume that $y \in \inter \Delta_C(x,\arclength_C(x,z))$.
Similarly, we may assume that $y \in \inter \Delta_C(z,\arclength_C(x,z))$.
For any $0 < \tau < \arclength_C(x,z)$, let $L(\tau)$ denote the line for which $L(\tau) \cap \bd [x,z]_C = \{ c,d \}$
with $\arclength_C(x,c) = \arclength_C(x,d) = \tau$, and let $R$ be the region
swept through by the lines $L(\tau)$, $0 < \tau< \arclength_C(x,z)$.
Since $C$ is smooth, the limit $L_+$ of $L(\tau)$ as $\tau$ approaches $\arclength(x,z)$ is the supporting line of $\Delta_C(x,\arclength(x,z))$ at $z$.
Similarly, if $\tau \to 0$, the limit $L_-$ of $L(\tau)$ is the supporting line of $\Delta_C(z,\arclength_C(x,z))$ at $x$.
Since $\Delta_C(x,\arclength_C(x,z)) \cap \Delta_C(z,\arclength_C(x,z))$ lies between the parallel lines $L_+$ and $L_-$, and this open unbounded strip is clearly contained in $R$, $y \in L(\tau)$ for some value of $\tau$, and (3) follows from Corollary~\ref{cor:quasiconvex}.
\end{proof}

\begin{ex}
Let $C$ be the unit ball of the $l_\infty$ norm in $\Re^2$. Then
\[
\Delta(\rho) = \left\{
\begin{array}{ll}
\{ (x_1,x_2) \in \Re^2 : |x_1| + |x_2| \leq \rho \}, & \hbox{if } 0 < \rho \leq 2,\\
\{ (x_1,x_2) \in \Re^2 : |x_1|+ |x_2| \leq \rho, |x| \leq 2, |y| \leq 2 \}, & \hbox{if } 2 < \rho \leq 4.
\end{array}
\right.
\]
\end{ex}

In this example, for any two points $x,z$ with $\arclength_C(x,z) \leq 2$,
we have $\arclength_C(x,y) + \arclength_C(y,z) = \arclength(x,z)$ for any $y \in [x,z]_C$.
Furthermore, if we replace the corners of $C$ with small circle arcs, then
the boundary of the arc-distance balls of the obtained body consists 'almost only' of segments that are parallel to
the segments in $\bd \Delta(\rho)$.
Thus, we may create a smooth convex body and points $x,y,z$ with $y \notin \inter [x,z]_C$
such that $\arclength_C(x,y)+ \arclength_C(y,z)=\arclength_C(x,z)$.

We propose the following questions.

\begin{prob}
Can we drop the smoothness condition in part (3) of Theorem~\ref{thm:triangle}?
\end{prob}

\begin{prob}
Prove or disprove that if $C$ is strictly convex, then $\Delta_C(\rho)$ is strictly convex
for any $0 < \rho \leq \frac{1}{2} \perim_C C$.
\end{prob}

\begin{prob}
Prove or disprove that if $C$ is strictly convex, then the inequalities in (1) and (3) of Theorem~\ref{thm:triangle} are strict.
\end{prob}

\section{Carath\'eodory numbers}\label{sec:caratheodory}

Now we recall the notion of the Carath\'eodory number of a convexity space 
(cf. \cite{L82}, \cite{R70} and \cite{S75}).

\begin{defn}\label{defn:caratheodory}
Let $(X,\GG)$ be a convexity space (for the definition, see Section~\ref{sec:notation}).
The \emph{Carath\'eodory number $\Car \GG$ of $\GG$} is the smallest positive integer $k$ such that
for any  $V \subseteq X$ and $p \in \conv_\GG (V)$ there is a set $W \subseteq V$ with $\card W \leq k$ and with $p \in \conv_\GG (W)$. 
If no such positive integer exists, we set $\Car\GG= \infty$ .
\end{defn}

\begin{defn}
Let $C \subset \Re^n$ be a convex body, and let $\GG_1$ (respectively $\GG_2$) be the family of closed
$C$-spindle convex sets (respectively, the $C$-ball convex sets) in $\Re^n$.
Then we call $\Car \GG_1$ (respectively $\Car\GG_2$) the \emph{spindle Carath\'eodory number} (respectively,
\emph{ball Carath\'eodory number}) of $C$, and denote it by $\Car^s C$ (respectively, by $\Car^b C$).
\end{defn}

These numbers were determined in \cite{BLNP07} for the Euclidean ball $B^n$ as $C$.

\begin{thm}
Let $C \subset \Re^2$ be a plane convex body. If $C$ is a parallelogram, then both
Carath\'eodory numbers of $C$ are two, otherwise both are three.
\end{thm}

\begin{proof}
By Theorem~\ref{thm:separation2d}, for any $C \subset \Re^2$, a closed set is $C$-spindle convex
if, and only if it is $C$-ball convex.
Thus, the two Carath\'eodory numbers of $C$ are equal.

Let $X \subset \Re^2$ be any closed set. The $C$-ball convex hull of $X$ is the intersection of
all the translates of $C$ that contain $X$.
If $X$ is a singleton, the assertion immediately follows, and thus we assume that $\card X > 1$.
If no translate of $C$ contains $X$, then, by Helly's theorem, there are at most three points of $X$
that are not contained in any translate of $C$.

Assume that there is a translate of $C$ containing $X$.
Consider a point $p \in \bd \sconv{X}$. If $p \in X$, we are done.
Assume that $p \notin X$.
We leave it to the reader to show that there is a translate $u+C$ with the property that
$\bd (u+C)$ contains $p$ and two distinct points $z_1, z_2 \in X$, such that the $C$-distance of $z_1$ and $z_2$ is 2, or,
if the $C$-distance of $z_1$ and $z_2$
is less than two, then the connected component of $(u+C) \setminus [z_1,z_2]$ containing $p$ does not contain points
at $C$-distance two. Without loss of generality, we may assume that the open arc in $\bd (u+C)$
with endpoints $z_1$ and $z_2$ that contains $p$ is disjoint from $X$.

If the $C$-distance of $z_1$ and $z_2$ is less than two, then, clearly, $p$ is contained in any translate of $C$
that contains $z_1$ and $z_2$, or, in other words, $p \in [z_1,z_2]_C$.
If the $C$-distance of $z_1$ and $z_2$ is equal to two, there are two parallel lines $L_1$ and $L_2$ that
support $u+C$ at $z_1$ and $z_2$, respectively.
Since $y \in \sconv{X}$, we have for $i=1,2$ that $z_i$ is the endpoint of the segment
$L_i \cap (u+C)$ closer to $p$.
Thus, $p \in [z_1,z_2]_C$.

Now consider a point $p \in \inter \sconv{X}$.
Let $v \in X$ be arbitrary. Choose a point $z \in \bd \sconv{X}$ such that $p \in [v,z]$.
Then, by the previous paragraph, there are points $z_1, z_2 \in X$ such that $z \in [z_1,z_2]_C$,
and clearly, $p \in \sconv{\{ v,z_1,z_2\}}$.

Finally, assume that $C$ is not a parallelogram.
Note that in that case there are three smooth points in $\bd C$ such that the unique
lines supporting $C$ at them are the sidelines of a triangle containing $C$.
Let $X$ be the vertex set of this triangle.
Observe that for a sufficiently large $\lambda > 0$, the centroid $c$ of $X$ is not contained
in the $(\lambda C)$-spindles determined by any two points of $X$.
Since $c \in \conv (X) \subseteq \sconv[\lambda C]{X}$ for any value of $\lambda$, we have that
both Carath\'eodory numbers of $C$ are three.
The observation that they are two if $C$ is a parallelogram follows from the next theorem.
\end{proof}

\begin{prop}\label{prop:cubecaratheodory}
If $C$ is an $n$-dimensional parallelotope, then $\Car^s C = \Car^b C = n$.
\end{prop}

\begin{proof}
Since both Carath\'eodory numbers are affine invariant quantities, we may assume that
$C$ is the unit ball of the $l_\infty$ norm.
By Corollary~\ref{cor:cubeconvexhull}, the two Carath\'eodory numbers of $C$ are equal.
Let $X$ be a closed set in $\Re^n$.
Then $\sconv{X}$ is either $\Re^n$ or the minimal volume axis-parallel box
that contains $X$.
In the first case it is easy to see that there are two points of $X$ that are not contained in any translate of $C$,
from which the assertion readily follows.

Assume that $\sconv{X}$ is the minimal volume axis-parallel box that contains $X$,
and consider a point $y \in \sconv{X}$.
Then, by a theorem of Lay \cite{L80}, there is a subset $X' \subseteq X$ with $\card X' \leq n$ such that
the minimal volume axis-parallel box containing $X'$ contains $y$.
Thus, we obtain $y \in \sconv{X'}$ and hence $\Car^s C \leq n$.

On the other hand, let $X$ be the set of those $n$ vertices
of $C$ which are connected by an edge of $C$ with a given vertex $v$ of $C$.
Then $\sconv{X} = C$.
Let $y$ be the vertex of $C$ opposite of $v$, and observe that, removing any point of $X$,
the $C$-ball convex hull of the remaining points is a facet of $C$ not containing $y$.
Thus, $\Car^s C \geq n$.
\end{proof}

Similarly to Corollary~\ref{cor:cubeconvexhull} and Proposition~\ref{prop:cubecaratheodory},
we can prove the following.

\begin{prop}\label{prop:simplex}
Let $C \subset \Re^n$ be a simplex.
\begin{itemize}
\item[(i)] For any closed set $X \subset \Re^n$, both $\bconv{X}$ and $\sconv{X}$ are either $\Re^n$,
or the smallest positive homothetic copy of $C$ that contains $C$.
\item[(ii)] $\Car^s C = \Car^b C = n+1$.
\end{itemize}
\end{prop}

We apply the following remark several times in this section.

\begin{rem}\label{rem:bconvcheck}
For any set $H\subset \Re^n$, a point $p\in\Re^n$, and a convex body $C\subset\Re^n$,
we have that $p\in\Bm(H)$ if, and only if, $C+p\supseteq H$. Similarly,
$p\in\Bp (H)$ if, and only if, $-C+p\supseteq H$.
\end{rem}

\begin{prop}\label{prop:BBcaratheodory}
Let $C\subset\Ren$ be an $n$-polytope with $k$ facets. Then the ball Carath\'eo\-dory number of
$C$ is at most $kn$.
\end{prop}

\begin{proof}
Let $X\subset\Re^n$ and $p\in\bconv{X}$ be given. We need to find a subset $Y$ of $X$ of cardinality $kn$
such that $p\in\bconv{Y}$. By Remark~\ref{rem:bconvcheck}, we have that $C\supseteq \Bm (X)$, that is,
$(\Re^n\setminus C)\cap \bigcap_{v\in X}(-C+v)=\emptyset$. Since $\Re^n\setminus C$ is the union of $k$ convex sets,
the statement follows from Helly's theorem.
\end{proof}

The following construction, similar to Example 4 in \cite{NT10}, shows that in $\Re^3$ there are convex bodies
with arbitrarily large ball Carath\'eodory numbers.

\begin{thm}\label{thm:BB}
Let $k\in\Ze^+$ be given. Then there is an \osymm convex body $C \subset \Re^3$ and a set $X\subset\Re^3$
such that $\origo\in\bconv{X}$, but for any $Y\subset X$, $\card Y< k$ we have that $o\notin \bconv{Y}$.
\end{thm}

\begin{proof}
We may assume that $k$ is even. Consider the paraboloid $P$ in $\Re^3$
defined by $x_3 = x_1^2 + x_2^2$.
We choose $k$ points on the parabola $P_0=P\cap\{x_1=0\}$, and
number these points according to the order in which they lie on the parabola:
$U=\{u_1,u_2,\ldots,u_k\}\subset P_0$.
Plane sections of $P$ parallel to the $x_2x_3$-plane
are translates of this parabola.
Let $P_i=P\cap\{x_1=i\}$ for $i\in I$, where
$I=\left\{-\frac{k}{2},-\frac{k}{2}+1,\ldots,-2,-1,1,2,\ldots,\frac{k}{2}\right\}$.
Now, for each $i\in I$, there is a unique translation vector $t_i$ such that
$P_0=P_i+t_i$. Let $U_i=(U\setminus\{u_{i}\})-t_i\subset P_i$.
Let $h>0$ be larger than the largest $x_3$
coordinate of the points in any $U_i$.
Finally, consider the (bounded) arc of $P_0$ that lies in the half-space
$\{x_3\leq h\}$. Delete from this arc very small open arcs around
each point of $U$, and call the remaining part of $P_0$ (the union of $k+1$ closed bounded arcs) $U_0$.
We define $C$ as the following \osymm convex body:
\[C=\conv\left[\left(\bigcup_{i=0}^k U_i-(0,0,h)\right)\bigcup-\left(\bigcup_{i=0}^k U_i-(0,0,h)\right)\right].\]

Let $X=\{t_i: i\in I\}$. Now, $\BC(X)\subseteq\BC\left\{t_{-\frac{k}{2}},t_{\frac{k}{2}}\right\}$, and the latter
is contained in the $x_1=0$ plane.
Moreover, $\BC(X)$ is contained in the planar region $\conv(P_0-(0,0,h))\cap -\conv(P_0-(0,0,h))$.
If the open arcs in the definition of $U_0$ are sufficiently small then a little more is true:
$\BC(X)$ is contained in the planar region $\conv(U_0-(0,0,h))\cap -\conv(U_0-(0,0,h))$.
Thus, $\BC(X)\subset C$. It follows,
by Remark~\ref{rem:bconvcheck}, that $o\in\BC\BC(X)=\bconv{X}$.
On the other hand, for any $i\in I$, we have that
$u_i\in \BC(X\setminus\{t_i\})$, and hence, $C\not\supset\BC(X\setminus\{t_i\})$.
It follows (again by Remark~\ref{rem:bconvcheck}) that $o\not\in\bconv{X\setminus\{t_i\}}$.
\end{proof}

This example may be modified in several ways.
First, we may generalize it for $\Re^n$ with $n>3$, by replacing $C$ with $C\times[-1,1]^{n-3}$ and leaving $X$ unchanged.
Second, by ``smoothening'' $C$, we may obtain a smooth and strictly convex
\osymm body $C$ in $\Ren$ with an arbitrarily large ball Carath\'eodory number.
Third, we may replace $U_0$ by a sufficiently dense finite subset of $U_0$, and thus obtain a polytope as $C$.
Finally, the following modification of the example yields an \osymm convex body in $\Re^3$ whose
ball Carath\'eodory number is infinity: In the construction, replace the finitely many $P_i$s
by planar sections of $P$ of the form $P_i=P\cap\{x_1=a_i\}$ where
$a_1,a_2,\ldots$ is a sequence of real numbers in $(-1,1)$, which is symmetric about zero and which does not contain any of its accumulation points.
Then, construct $U$ similarly: let $u_1,u_2,\ldots$
be a bounded sequence of points on $P_0$, which does not contain any of its accumulation points.

\begin{prob}
For $n \geq 3$, find the minima of the ball/spindle Carath\'eodory numbers of the $n$-dimensional convex bodies,
and if it exists, find the maximum of their spindle Carath\'eodory numbers.
\end{prob}

\begin{prob}
For $n \geq 3$, prove or disprove the existence of a convex body $C \subset \Re^n$ such that its two Carath\'eodory numbers are different.
\end{prob}

\begin{prob}
It is known (cf. \cite{BLNP07}) that both Carath\'eodory numbers of the $n$-dimensional Euclidean ball are $n+1$.
Prove or disprove that this holds also in a small neighborhood of the Euclidean ball. If the answer is negative, is the set of ball (resp., spindle)
Caratheodory numbers bounded from above in a neighborhood of the Euclidean ball?
\end{prob}

\section{Finitely generated $C$-ball convex sets}\label{sec:ballspindle}

Clearly, every $C$-spindle is $C$-ball convex, but the converse is
not true in general. However, there are convex bodies $C$ for which
every $C$-ball convex set is a $C$-spindle, such as the simplices and rectangular boxes of $\Ren$, see Corollary~\ref{cor:cubeconvexhull} and
Proposition~\ref{prop:simplex}.
In this section, we examine a more general problem: We investigate those convex bodies $C$ that have the property that every $C$-ball convex set is the $C$-ball convex hull of finitely many points. 

\begin{defn}\label{defn:k-generated}
If every $C$-ball convex set is the $C$-ball convex hull of at most $k$ points, for some fixed $k\geq 2$, then we say that 
\emph{every $C$-ball convex set is $k$-generated}. Similarly, if every $C$-ball convex set is the $C$-ball 
convex hull of finitely many points, then we say that \emph{every $C$-ball convex set is finitely generated}.
\end{defn}

Note that if every $C$-ball convex set is $k$-generated for some convex body $C\subset\Ren$,
where $k\geq 2$, and $P$ is a $C$-ball convex set, $\card (P)>1$, then $P$ can
be obtained as the $C$-ball convex hull of exactly $k$ points of $\Ren$.

\begin{thm}\label{thm:kn}
If $C$ is a convex body in $\Ren$ for which every $C$-ball convex set is $k$-generated, then
$C$ is an $n$-polytope with at most $kn$ facets.
\end{thm}

This theorem is a consequence of the next three lemmas.
The first of these readily follows  by the minimality property of $C$-ball convex hull as the intersection of translates of $C$.

\begin{lem}\label{lem:Cballconvex}
Let $C$ be a convex body in $\Ren$. Let $K\subseteq \Ren$ be a bounded set,
and assume that $K=\bconv{S}$ for some closed set $S\subseteq K$.
Then, a translate $C+v$ contains $K$ and fulfils $\bd(C+v) \cap K\neq \emptyset$ if and only if $\bd(C+v) \cap S\neq \emptyset$.
\end{lem}

\begin{lem}\label{lem:kn}
If $C\subset\Ren$ is an $n$-polytope for which every $C$-ball convex set is $k$-generated, then $C$ has at most $kn$ facets.
\end{lem}

\begin{lem}\label{lem:polytope}
If every $C$-ball convex set is finitely generated for a convex body $C\subset \Ren$, then $C$ is an $n$-polytope.
\end{lem}

\begin{proof}[Proof of Lemma~\ref{lem:kn}]
We can obtain $\frac{1}{2}C$ as the intersection of finitely many translates of $C$:
$\frac{1}{2}C=\bigcap_{i=1}^m (C+v_i)$.
We perturb the translation vectors $v_i$ to obtain another polytope
$P=\bigcap_{i=1}^m (C+w_i)$ in a way that
$|v_i-w_i|<\epsilon$ for a sufficiently small $\epsilon>0$, and $P$ is a simple polytope (that is, every vertex of $C$ is contained in exactly $n$ facets of $C$),
and each facet of $P$ is contained in the relative interior of a facet of $C+w_i$ for some value of $i$, and the $n_i$s are pairwise distinct.
By Lemma~\ref{lem:Cballconvex}, there are $k$ points of $P$ such that every facet of $P$ contains at least one of these points.
Since any point is contained in at most $n$ facets of $P$, we obtain $m\leq kn$.
\end{proof}

\begin{proof}[Proof of Lemma~\ref{lem:polytope}]
Assume that $C\subset\Ren$ is a convex body which is not a polytope, and let $0<r<1$. It follows from
Theorems~2.2.4 and 2.2.9 of \cite{Sch93} that there is an infinite sequence of
pairwise distinct triples $T_i=(p_i,H_i,n_i)$, $i=1,2,3,\dots$, such that $p_i$
is a smooth boundary point of $rC$, $H_i$ is the unique supporting hyperplane of
$rC$ at $p_i$, and $n_i$ is the outer normal unit vector of $H_i$ with respect
to $rC$, for which $n_i\neq n_j$ if $i\neq j$. To see this directly, it is also
easy to construct such triples applying induction.

We choose an infinite subsequence $\{T_i\mid i\in I\}$ of the triples, for which
there is at most one cluster point of $\{p_i\mid i\in I\}$ and $\{n_i\mid i\in
I\}$, resp., and that cluster point is not equal to any $p_i$ and $n_i$.

Let $H_i^+$ be the half-space determined by $H_i$ which contains $rC$. Let $r<r'<1$. Let $C_i$ be that translate of $r'C$ which touches both $rC$ and $H_i$ at $p_i$, $i\in I$. Let $P=\bigcap_{i\in I} C_i$. Since there is no cluster point among the points $p_i$ and vectors $n_i$ ($i\in I$), we can define a sequence $B_i$ of balls of positive radii such that for every $i\in I$, $B_i$ is centered at $p_i$, and $B_i\subseteq \inter (C_j)$ for any $j\neq i$. Therefore, every $C_i$ can be translated by a vector $v_i$ towards the direction $n_i$ within a sufficiently small, but positive distance such that the translates $C'_i=C_i+v_i$, $i\in I$ form $P'=\bigcap_{i\in I} C'_i$ in such a way that $H'_i=H_i+v_i$ is the unique supporting hyperplane of $P'$ at $p'_i=p_i+v_i$, and at most $n$ hyperplanes have a common point among the hyperplanes $H_i$ ($i\in I$).

Then, since the smaller homothetic copies $C_i$ and $C'_i$ ($i\geq 1$) of $C$ are $C$-ball convex sets by Corollary~\ref{cor:addition}, by assumption, $P'$ is a $C$-ball convex hull of finitely many points, so there should be a finite subset of $\bd(P')$ such that every face $H'_i\cap P'_i$ contains at least one element of $S$ (otherwise there would be a translate of $C$ touching $P'$ and having a disjoint boundary from those points, and by Lemma~\ref{lem:Cballconvex}, that
would contradict the fact that $P'$ is a $C$-ball convex hull of those points). But one point can be contained in at most $n$ faces among the infinitely many faces $H'_i\cap P'_i$, so $P'$ can not be the $C$-ball convex hull of finitely many points.
\end{proof}

Theorem \ref{thm:kn} implies the following corollary.

\begin{cor}\label{cor:2n}
Let $C$ be a convex body in $\Ren$. If every $C$-ball convex
set in $\Ren$ is a $C$-spindle, then $C$ is a polytope and it has at most $2n$ facets.
\end{cor}

While Corollary~\ref{cor:2n} is sharp for parallelotopes, we do not know if it is
the case for Theorem~\ref{thm:kn} for $k\geq 3$. We ask the following question.

\begin{prob}
Let $k \geq 3$ be arbitrary. Is there an $n$-polytope $C\subset\Ren$ with $kn$ facets such that every $C$-ball convex set is $k$-generated?
\end{prob}

In the first part of this section, we find an upper bound for the number of
facets of those polytopes $C$ 
for which every $C$-ball convex set is $k$-generated.
It is also natural to estimate these numbers from below.
Now, we consider the following problem: For a fixed integer $k\geq 2$, what is
the maximum number $m=m(n,k)$ for which every $n$-polytope $C$ that has at
most $m$ facets, also has the property that every $C$-ball convex set is
$k$-generated? We have the following partial solution for this problem.

\begin{thm}\label{thm:lowerbound}
Let $n  \geq 3$.
\begin{itemize}
\item[(1)] If $2 \leq k \leq n$, then there is an $n$-polytope $C\subset\Ren$ having $n+k+2$ facets such that not every $C$-ball convex set is $k$-generated.
\item[(2)] If $k \geq 2$, and $C\subset\Ren$ is any $n$-polytope having at most $n+k+1$ facets, then every $C$-ball convex set is $k$-generated.
\end{itemize}
\end{thm}

The example of a pentagon shows that the assertion in (2) fails for $n=k=2$. It is easy to see that (2) also holds for $n=2$ and $k\geq 3$.

\begin{proof}
To prove (1), assume $2\leq  k\leq n$, and let $C\subset\Ren$ be the $n$-polytope which is obtained
from an $n$-simplex $S_n$ by intersecting it with 
$k$ closed half-spaces near $k$ vertices of $S_n$ so that the $k$ new facets are pairwise disjoint. 
Let $0<r<1$ arbitrary. Observe that for any set $T$ of $k$ points, $rC$ has a facet disjoint from $T$. Thus, by Lemma~\ref{lem:Cballconvex}, $rC$ is not the $C$-ball convex hull of at most $k$ points.

Next, we prove (2). Now assume that $k \geq 2$ and $C\subset\Ren$ is an $n$-polytope with $n+k+1$ facets.
Let $P$ be an arbitrary $C$-ball convex set of $\Re^n$, $P\neq \Re^n$. Then $P$ is a polytope of at most $n+k+1$ facets.
We will assume that $C$ is a simple $n$-polytope. We may do so, since it is easy
to see that for every $C$-ball convex set $P$ there is a sequence
$\{P_i\}_{i=1}^{\infty}$ of simple $n$-polytopes such that every $P_i$ is a
$C_i$-ball convex set, and $P_i\to P$, $C_i\to C$ in the Hausdorff metric, as
$i\to \infty$, where $C_i\subset\Ren$ is an $n$-polytope for every $i=1,2,3,\dots$.
Then any sequence of at most $k$ element subsets which span $P_i$ as a $C_i$-ball convex
hull, for $i=1,2,3,\dots$, has a subsequence whose elements converge to an at most $k$
element subset of $P$. Clearly, the $C$-ball convex hull of 
this set is $P$. By a similar limit argument, we may further assume that $\dim
P=n$, $P$ is a simple polytope, and $P$ has exactly $n+k+1$ facets.

We need to show that there are $k$ points whose $C$-ball convex hull is $P$.
By Lemma~\ref{lem:Cballconvex}, it is sufficient to prove that there are $k$ points of $P$ such that
for each facet $F$ of $C$ there is a translate $C+t$ of $C$ ($t\in\Ren$) that contains $P$ and for which $F+t$ contains at least one of the $k$ points.
With the above assumption on the number of facets and dimension of $P$, it is equivalent to the existence of a set $T$ of $k$ points such that every facet contains at least one element of $T$.

Let $v$ be a vertex of $P$. Let $H$ be a hyperplane such that $v\notin H$ and $H$ is parallel to a supporting hyperplane $H'$ of $P$ at $v$ for which $H'\cap P=\{v\}$. Let $\pi$ be the central projection of $\Ren\smallsetminus H'$ to $H$ from the point $v$, that is 
$\pi(x)=\textrm{aff}(v,x)\cap H$. Then $\pi(P)$ is bounded, in fact, it is an $(n-1)$-dimensional simplex $S_{n-1}$ since $P$ is a simple polytope. Consider the projection of the facets of $P$ under $\pi$. Then the images of those facets which contain $v$ are the facets of an $S_{n-1}$, and we denote them by $\mathcal F=\{F_1,F_2,\dots,F_n\}$. The images of the remaining $k+1$ facets form a tiling of $S_{n-1}$, we denote them by $\mathcal A=\{A_1,A_2,\dots,A_{k+1}\}$. 
Clearly, every $A_i$ is an $(n-1)$-polytope. Denote by $V$ the vertex set of $S_{n-1}$.

Obviously, there are $k$ points that span $P$ as the $C$-ball convex hull of those points if and only if either there are $k-1$ points of $S_{n-1}$ such that 
every $A_i$ contains at least one of them, or there are $k$ points of $S_{n-1}$ such that every $A_i$ and $F_j$ contains at least one of them. 
If there are three elements of $\mathcal A$ having a common point, then such a set of $k-1$ points exists: 
we pick an arbitrary point belonging to the intersection of the three elements of $\mathcal A$ 
and one point from each remaining element of $\mathcal A$. 
So, from now on, we may assume that there is no common point of three elements of~$\mathcal A$.

Since $\mathcal A$ contains more than one element, no $A_i$ contains all the vertices of $S_{n-1}$.

Now we show that there are two disjoint elements of $\mathcal A$ such that both contain at least one vertex of $S_{n-1}$. Let $A_1$ and $A_2$ be an intersecting pair, $A_i\cap V\neq \emptyset$, $i=1,2$.
Consider an $(n-2)$-dimensional affine subspace $H_{12}$
that separates them in $H$. Then $A_1\cap A_2=H_{12}\cap S_{n-1}$ since
$H_{12}\cap S_{n-1}$ is covered by $\mathcal A$, and it can not intersect
any element of $\mathcal A$ distinct from $A_1$ and $A_2$ (otherwise,
there would be a common point of $A_1\cap A_2$ and some $A_i$, for
$i\neq 1,2$). Thus we also obtain that $A_1\cap A_2$ is
$(n-2)$-dimensional. Now, if $p\in A_1$, $q\in A_2$ are arbitrary
points, then let $x=[p,q]\cap H_{12}$. We obtain $[p,x]\subseteq A_1$,
$[x,q]\subseteq A_2$, so $A_1\cup A_2$ is convex, and therefore it can
not contain all vertices of $S_{n-1}$. So there is an $A_j$, $j\neq 1,2$, say $A_3$, such that 
$A_3\cap V\neq \emptyset$, and
either $A_1$ or $A_2$,  say $A_1$, is disjoint from $A_3$, because $A_3$
is disjoint from $H_{12}$. So we found two disjoint elements of $\mathcal A$, $A_1$ and $A_3$, even if $A_1$ and $A_2$ were intersecting.

Let us choose two disjoint elements of $\mathcal A$, say $A_1$ and $A_2$, such that
$A_i\cap V\neq \emptyset$, $i=1,2$.
Since $S_{n-1}$ has $n\geq 3$ vertices, it follows that one of them, say $A_1$,
contains at most $n-2$ vertices of $S_{n-1}$.
Now, we pick a vertex $w$ in $A_2\cap V$ and an edge $E$ of $S_{n-1}$
connecting a vertex from $V\cap A_1$ and a vertex from $V\smallsetminus (A_1\cup \{w\})\neq \emptyset$.
Furthermore, we may pick an $A_i$, say $A_3$, such that $E\cap A_1\cap A_3\neq \emptyset$, since $E$ intersects $A_1$ but also has points outside $A_1$.
Finally, the desired $k$ points: Take a point $u\in E\cap A_1\cap A_3$,
take $w\in A_2$, and take further $k-2$ points,
one from each remaining $A_i$.
As $u$ is contained in $E$, it is contained in each facet of $S_{n-1}$ but the two that do not contain $E$. These two facets contain $w$, and thus, every element of $\mathcal F$ contains
at least one of the points. Since, clearly, the same holds for the elements of $\mathcal A$ as well, these points indeed satify the required property.
\end{proof}

We ask the following question.

\begin{prob}
For any given $n \geq 3$ and $k \geq 2$, find a geometric characterization of
those $n$-polytopes $C\subset\Ren$ for which every $C$-ball convex set is
$k$-generated. In particular, find a geometric characterization of those
$n$-polytopes $C$ for which every $C$-ball convex set is a $C$-spindle.
\end{prob}

\section{Stability of the operation $\Bp$ and covering intersections of balls}\label{sec:covering}

We consider the Levy-Markus-Gohberg-Boltyanski-Hadwiger Covering Problem
(also known as the Illumination Problem) for two families
of convex bodies, denoted by $\D$ and $\DT$ (see Definition~\ref{def:DD}, and Remark~\ref{rem:DD}), associated to any convex body $C$.

The covering number (see Definition~\ref{def:cover}) of sets of Euclidean
constant width (members of $\D[B^n]$) has been studied extensively.
One reason for its popularity is its connection to Borsuk's problem
on partitioning convex sets into sets of smaller diameter.
Weissbach \cite{W96} and Lassak \cite{L97} proved that the covering number
of a set of Euclidean constant width in $\Re^3$ is at most six.
In \cite{BLNP07}, this result is extended to any set $K$ obtained as
the intersection of Euclidean unit balls with the property that the
set of the centers is of diameter at most one (members of $\DT[B^n]$).
Recently, further bounds on the covering number of sets in $\DT[B^n]$ in dimensions $n=4,5$ and $6$
have been found by Bezdek and Kiss \cite{BK09}.
The best general bound is due to Schramm \cite{Sch88} who proved that
the covering number of a set of Euclidean constant width in $\Ren$ is at most
$5n^{3/2}(4+\log n)\left(\frac{3}{2}\right)^{n/2}$.
This result has been extended to members of $\DT[B^n]$ as well, cf. Bezdek \cite{B10}. 
For surveys on covering (illumination) see Bezdek \cite{B07}, Martini and Soltan \cite{MS99} and Boltyanski, Martini and Soltan \cite{BoMaSo}.

In this section, we study the stability of bounds on the covering number of convex sets in $\D$ and in $\DT$. First, we prove that the operation $\Bp$ is stable in a certain sense (Proposition~\ref{prop:limitofsets}), and then we deduce our stability results concerning covering numbers.

\begin{defn}\label{def:DD}
Let $C$ be a convex body in $\Ren$. Let
\[
\D=\{K\subset\Ren : K=\Bp (K) \}, \mbox{ and}
\]
\[
\DT=\{K\subset\Ren : K=\Bp (X) \mbox{ for some } X\subset\Ren \mbox{ with } X\subseteq \Bp(X)\}.
\]
\end{defn}

\begin{rem}\label{rem:DD}
For any convex body $C$, we have $\D\subseteq\DT$.
Moreover, if $C=-C$ then $\D$ is the family of diametrically maximal sets of diameter one
in the Minkowski space (that is, finite dimensional real Banach space) with unit ball $C$.
Since in the Euclidean space, a convex set is diametrically maximal if, and only if,
it is of constant width, it follows that $\D[B^n]$ is the \emph{family of sets of Euclidean constant width one}.
\end{rem}

\begin{defn}\label{def:cover}
 Let $K\in\Re^n$ be a convex body. The \emph{covering number} (also called the \emph{illumination number}) $i(K)$ of $K$ is the
 minimum number of positive homothetic copies of $K$, with homothety ratio less than one,
 that cover $K$. For a family $\FF$ of convex bodies, we set $i(\FF)=\max\{i(K) : K\in\FF\}$.
 We note that the illumination number is usually defined via the notion of illumination by directions
 (or light sources), which we do not follow here --- the equivalence of those definitions with the one given here is well known, cf. \cite{B07}.
\end{defn}

The covering number is invariant under non-singular affine transformations, thus it is natural to use the Banach-Mazur distance to compare two convex bodies $K$ and $L$:
\[
d(K,L)=\inf\{\lambda>0 : \ K \subset T(L) \subset \lambda(K)\},
\]
where the infimum is taken over all non-singular affine transformations $T$.
Recall that this distance is multiplicative (the triangle inequality holds
with multiplication instead of addition) and
the distance of a convex body from any non-singular affine image
of itself is one.


We phrase, informally, the problem of the stability of the covering number in the following two ways.

\emph{Question 1.}
Fix a convex body $C$. If $K$ is `close' to a set $L\in\D$ (resp., to a set $L\in\DT$),
does it follow that $i(K)\leq i(\D)$ (resp.,$i(K)\leq i(\DT)$)?

\emph{Question 2.}
Fix a convex body $C$. If $D$ is `close' to $C$, does it follow that $i(\D[D])\leq i(\D)$ and
$i(\DT[D])\leq i(\DT[C])$?

Recall that $i(K)=n+1$ for any smooth convex body, while $i([0,1]^n)=2^n$ for the cube.
It illustrates that the covering number, in general, may vary
significantly along arbitrarily small perturbations.
However, Theorems~\ref{thm:easyStability} and \ref{thm:hardStability} provide
positive answers to both questions.

\begin{thm}\label{thm:easyStability}
For every $n\in\Ze^+$ and for every convex body $C\subset\Ren$, there is a $\delta>1$ such that if $d(K,L)<\delta$
for a convex body $K$ in $\Ren$ and $L\in\D$ (resp., $L\in\DT$),
then $i(K)\leq i(\D)$ (resp., $i(K)\leq i(\DT)$).
\end{thm}

\begin{thm}\label{thm:hardStability}
For every $n\in\Ze^+$ and for every convex body $C\subset\Ren$, there is a $\delta>1$ such that if $d(C,D)<\delta$
for a convex body $D\subset\Ren$, then
\begin{enumerate}
\item $i(\D[D])\leq i(\D[C])$  and
\item $i(\DT[D])\leq i(\DT[C])$.
\end{enumerate}
\end{thm}

The main tool in proving these results is the
following observation that shows that the operation $\Bp$ (and, similarly, $\Bm$) is stable in a certain sense.

\begin{prop}\label{prop:limitofsets}
Let $C_1,C_2,\ldots$ be a sequence of convex bodies in $\Ren$ converging to a convex body $C$ in the metric space of closed convex subsets of $\Ren$ equipped with the Hausdorff metric. Let $X_1,X_2,\ldots$ be a sequence of closed sets in $\Ren$ converging to a set $X$ such that the sequence $\Bp[C_i](X_i)$ also converges (to some set $K$). Assume that $\inter K\neq\emptyset$
or $\inter(\Bp(X))\neq\emptyset$. Then $K=\Bp(X)$.
\end{prop}

\begin{proof}
First, we show that $\inter(\Bp(X))\subseteq K$. Let $u\in\Ren\setminus K$.
Then, for infinitely many $k\in\Ze$, there is a $q_k\in X_k$ such that $u\notin q_k+C$.
By taking a subsequence, we may assume that the $q_k$s converge to a point, say $q$. Clearly, $q\in X$. 
Moreover, we have $u\notin q+\inter C$. Thus, $u\notin\inter(\Bp(X))$.

Next, we show that $\inter K \subseteq \Bp(X)$. Let $u\in\inter K$. Then, there is a $\delta>0$ such that, for all sufficiently large $n\in\Ze$, $u+\delta B^n \subset\Bp[C_n](X_n)$. It follows that,
for all sufficiently large $n\in\Ze$, $u\in\Bp[C_n](X)$. Hence, $u\in\Bp(X)$.
\end{proof}

We leave it as an exercise to show that the condition
$\inter(\ldots)\neq\emptyset$ cannot be removed.

Let $\KK^n$ denote the space of affine equivalence classes of convex bodies in $\Ren$ endowed with the Banach-Mazur distance. We show that
$i:\KK^n\to\Ze^+$ is upper semi-continuous.

\begin{prop}\label{prop:semicont}
Let $K\in\KK^n$. Then there is a $\delta>1$ such that
$i(L)\leq i(K)$ for any $L\in\KK^n$ with $d(K,L)<\delta$.
\end{prop}

\begin{proof}
Let $m=i(K)$. Then $K$ is covered by $v_1+\lambda K,\ldots,v_m+\lambda K$ for some $0<\lambda<1$ and some translation vectors $v_1,\ldots,v_m$.
Let $\delta=1/\sqrt{\lambda}$. We may assume that $L\subseteq K\subseteq \delta L$. Then $L$ is covered by $v_1+\sqrt{\lambda} L,\ldots,v_m+\sqrt{\lambda}L$.
\end{proof}


\begin{rem}\label{rem:dbody}
Let $C\subset\Ren$ be a convex body and $X\subset\Ren$ a set for which $X\subseteq\Bp(X)$. Then $\inter\left( \Bp(X)\right)\neq\emptyset$. To see this, one may assume that $X$ is convex, and then show that any point in the relative interior of $X$ is an interior point of $\Bp(X)$.
It follows that the members of $\D$ and $\DT$ are bodies.
\end{rem}

\begin{proof}[Proof of Theorem~\ref{thm:easyStability}]
By the semi-continuity of $i$ on $\KK^n$ (Proposition~\ref{prop:semicont}) and the compactness of $\KK^n$, it is sufficient to show that
$\DT$ and $\D$ are closed subsets of $\KK^n$.
Let $K_1,K_2,\ldots$ be a convergent sequence of convex bodies in $\DT$. Then, by John's theorem, each one has an affine image $K_i'$ such that
$B^n\subseteq K_i'\subseteq nB^n$. Now, $K_1',K_2',\ldots$ is a sequence of convex bodies in $\DT$ which is convergent
with respect to the Hausdorff distance. By Proposition~\ref{prop:limitofsets}, the limit is also in $\DT$. The statement for $\D$ easily follows.
\end{proof}

\begin{proof}[Proof of Theorem~\ref{thm:hardStability}]
By Theorem~\ref{thm:easyStability}, we need to prove that for any $\varepsilon>1$ there is a $\delta>1$ such that if
$d(C,D)<\delta$, then for every $L\in\D[D]$ there is a
$K\in\D$ with $d(K,L)<\varepsilon$.

Let $D_k$ be a sequence of convex bodies in $\Ren$ such that $d(C,D_k)<1+\frac{1}{k}$ and let $X_k$ be a sequence of sets such that $L_k:=\Bp[C_k](X_k)\in\D[C_k]$. 
Suppose, for contradiction, that for each
$L_k$ the closest member of $\D$ is of distance at least $\mu>1$.
By compactness, we may choose a convergent subsequence of the $X_k$s.
By taking a subsequence again, we may assume that the $L_k$s converge, too.
Now, by Proposition~\ref{prop:limitofsets} (using John's theorem, as in the proof of Theorem~\ref{thm:easyStability}), the limit of these $L_k$s is a member of $\D$, a contradiction.
\end{proof}

\noindent {\bf Acknowledgements.} We are grateful to Antal Jo\'os and Steven Taschuk
for the valuable conversations that
we had with them on various topics covered in these notes.

\bibliographystyle{amsplain2}
\bibliography{biblio}

\providecommand{\bysame}{\leavevmode\hbox to3em{\hrulefill}\thinspace}
\providecommand{\MR}{\relax\ifhmode\unskip\space\fi MR }
\providecommand{\MRhref}[2]{%
  \href{http://www.ams.org/mathscinet-getitem?mr=#1}{#2}
}
\providecommand{\href}[2]{#2}
\begin{thebibliography}{10}

\bibitem{BK09}
K.~Bezdek and G.~Kiss, \emph{On the {X}-ray number of almost smooth convex
  bodies and of convex bodies of constant width}, Canad. Math. Bull.
  \textbf{52} (2009), no.~3, 342--348.

\bibitem{B07}
K.~Bezdek, \emph{The illumination conjecture and its extensions}, Period. Math.
  Hungar. \textbf{53} (2006), no.~1-2, 59--69.

\bibitem{B10}
K.~Bezdek, \emph{Classical topics in discrete geometry}, CMS Books in
  Mathematics, Springer, to appear.

\bibitem{BCCs06}
K.~Bezdek, R.~Connelly, and B.~Csik{\'o}s, \emph{On the perimeter of the
  intersection of congruent disks}, Beitr\"age Algebra Geom. \textbf{47}
  (2006), no.~1, 53--62.

\bibitem{BLNP07}
K.~Bezdek, Z.~L{\'a}ngi, M.~Nasz{\'o}di, and P.~Papez, \emph{Ball-polyhedra},
  Discrete Comput. Geom. \textbf{38} (2007), no.~2, 201--230.

\bibitem{BN06}
K.~Bezdek and M.~Nasz{\'o}di, \emph{Rigidity of ball-polyhedra in {E}uclidean
  3-space}, European J. Combin. \textbf{27} (2006), no.~2, 255--268.

\bibitem{B55}
L.~Bieberbach, \emph{Zur {E}uklidischen {G}eometrie der {K}reisbogendreiecke},
  Math. Ann. \textbf{130} (1955), 46--86.

\bibitem{B70}
L.~Bieberbach, \emph{Zwei {K}ongruenzs\"atze f\"ur {K}reisbogendreiecke in der
  {E}uklidischen {E}bene}, Math. Ann. \textbf{190} (1970/1971), 97--118.

\bibitem{BoMaSo}
V.~Boltyanski, H.~Martini, and P.~S. Soltan, \emph{Excursions into
  {C}ombinatorial {G}eometry}, Universitext, Springer-Verlag, Berlin, 1997.

\bibitem{E65}
H.~G. Eggleston, \emph{Sets of constant width in finite dimensional {B}anach
  spaces}, Israel J. Math. \textbf{3} (1965), 163--172.

\bibitem{FM82}
I.~F\'ary and E.~Makai, Jr., \emph{Isoperimetry in variable metric}, Studia
  Sci. Math. Hungar. \textbf{17} (1982), 143--158.

\bibitem{G32}
S.~Go\l\c{a}b, \emph{Some metric problems of the geometry of {M}inkowski},
  Trav. Acad. Mines Cracovie \textbf{6} (1932), 1--79.

\bibitem{Gr57}
B.~Gr\"unbaum, \emph{A proof of {V}\'azsonyi's conjecture}, Bull. Research
  Couincil Israel \textbf{A 6} (1956), 77--78.

\bibitem{He56}
A.~Heppes, \emph{Beweis einer {V}ermutung von {A}. {V}\'azsonyi}, Acta Math.
  Acad. Sci. Hungar. \textbf{7} (1956), 463--466.

\bibitem{HR56}
A.~Heppes and P.~R{\'e}v{\'e}sz, \emph{A splitting problem of {B}orsuk}, Mat.
  Lapok \textbf{7} (1956), 108--111.

\bibitem{KW71}
D.~C. Kay and E.~W. Womble, \emph{Axiomatic convexity theory and relationships
  between the {C}arath\'eodory, {H}elly, and {R}adon numbers}, Pacific J. Math.
  \textbf{38} (1971), 471--485.

\bibitem{K02}
W.~Kubis, \emph{Separation properties of convexity spaces}, J. Geom.
  \textbf{74} (2002), 110--119.

\bibitem{KMP10}
Y.~S. Kupitz, H.~Martini, and M.~A. Perles, \emph{Ball polytopes and the
  {V}\'azsonyi problem}, Acta Math. Hungar. \textbf{126} (2010), no.~1-2,
  99--163.

\bibitem{KMW96}
Y.~Kupitz, H.~Martini, and B.~Wegner, \emph{A linear-time construction of
  {R}euleaux polygons}, Beitr\"age Algebra Geom. \textbf{37} (1996), no.~2,
  415--427.

\bibitem{KMP05}
Y.~Kupitz, H.~Martini, and M.~Perles, \emph{{Finite sets in $\mathbb R^d$ with
  many diameters -- a survey.}}, {Proceedings of the International Conference
  on Mathematics and its Applications, ICMA-MU 2005, Bangkok, Thailand,
  December 15--17, 2005. Bangkok: Mahidol University. 91-112 (2005).}, 2005.

\bibitem{LN08}
Z.~L{\'a}ngi and M.~Nasz{\'o}di, \emph{Kirchberger-type theorems for separation
  by convex domains}, Period. Math. Hungar. \textbf{57} (2008), no.~2,
  185--196.

\bibitem{L82}
M.~Lassak, \emph{Carath\'eodory's and {H}elly's dimensions of products of
  convexity structures}, Colloq. Math. \textbf{46} (1982), no.~2, 215--225.

\bibitem{L91}
M.~Lassak, \emph{On five points in a plane convex body pairwise in at least
  unit relative distances}, Intuitive Geometry (Szeged, 1991), Colloq. Math.
  Soc. J\'anos Bolyai, vol.~63, North-Holland, Amsterdam, 1994, pp.~245--247.

\bibitem{L97}
M.~Lassak, \emph{Illumination of three-dimensional convex bodies of constant
  width}, Proceedings of the 4th International Congress of Geometry
  (Thessaloniki, 1996), Giachoudis-Giapoulis, Thessaloniki, 1997, pp.~246--250.

\bibitem{L80}
S.~R. Lay, \emph{Separating two compact sets by a parallelotope}, Proc. Amer.
  Math. Soc. \textbf{79} (1980), no.~2, 279--284.

\bibitem{MS09}
H.~Martini and M.~Spirova, \emph{On the circular hull property in normed
  planes}, Acta Math. Hungar. \textbf{125} (2009), no.~3, 275--285.

\bibitem{MaSw2}
H.~Martini and K.~J. Swanepoel, \emph{The geometry of {M}inkowski spaces---a
  survey. {II}}, Expo. Math. \textbf{22} (2004), no.~2, 93--144.

\bibitem{MS99}
H.~Martini and V.~Soltan, \emph{Combinatorial problems on the illumination of
  convex bodies}, Aequationes Math. \textbf{57} (1999), no.~2-3, 121--152.

\bibitem{M35}
A.~E. Mayer, \emph{Eine \"{U}berkonvexit\"at}, Math. Z. \textbf{39} (1935),
  no.~1, 511--531.

\bibitem{Mei11}
E.~Meissner, \emph{\"{U}ber punktmengen konstanter breite}, Vierteljahrsschr.
  Nat.forsch. Ges. Z\"urich \textbf{56} (1911), 42--50.

\bibitem{NT10}
M.~Nasz{\'o}di and S.~Taschuk, \emph{On the transversal number and
  {VC}-dimension of families of positive homothets of a convex body}, Discrete
  Math. \textbf{310} (2010), no.~1, 77--82.

\bibitem{AP95}
J.~Pach and P.~K. Agarwal, \emph{Combinatorial {G}eometry}, Wiley-Interscience
  Series in Discrete Mathematics and Optimization, John Wiley \& Sons Inc., New
  York, 1995, A Wiley-Interscience Publication.

\bibitem{R70}
J.~R. Reay, \emph{Caratheodory theorems in convex product structures}, Pacific
  J. Math. \textbf{35} (1970), 227--230.

\bibitem{S67}
J.~J. Sch{\"a}ffer, \emph{Inner diameter, perimeter, and girth of spheres},
  Math. Ann. 173 (1967), 59--79; addendum, ibid. \textbf{173} (1967), 79--82.

\bibitem{Sch93}
R.~Schneider, \emph{Convex bodies: the {B}runn-{M}inkowski {T}heory},
  Encyclopedia of Mathematics and its Applications, vol.~44, Cambridge
  University Press, Cambridge, 1993.

\bibitem{Sch88}
O.~Schramm, \emph{Illuminating sets of constant width}, Mathematika \textbf{35}
  (1988), no.~2, 180--189.

\bibitem{S75}
G.~Sierksma, \emph{Carath\'eodory and {H}elly-numbers of convex
  product-structures}, Pacific J. Math. \textbf{61} (1975), no.~1, 275--282.

\bibitem{Sie1984}
G.~Sierksma, \emph{Extending a convexity space to an aligned space}, Nederl.
  Akad. Wetensch. Indag. Math. \textbf{46} (1984), no.~4, 429--435.

\bibitem{Str57}
S.~Straszewicz, \emph{Sur un probl\`eme g\'eom\'etrique de {P}. {E}rd\"os},
  Bull. Acad. Polon. Sci. Cl. III. \textbf{5} (1957), 39--40, IV--V.

\bibitem{Val64}
F.~A. Valentine, \emph{Convex sets}, McGraw-Hill Series in Higher Mathematics,
  McGraw-Hill Book Co., New York, 1964.

\bibitem{vdV93}
M.~L.~J. van~de Vel, \emph{Theory of convex structures}, North-Holland
  Mathematical Library, vol.~50, North-Holland Publishing Co., Amsterdam, 1993.

\bibitem{W96}
B.~Wei{\ss}bach, \emph{Invariante {B}eleuchtung konvexer {K}\"orper},
  Beitr\"age Algebra Geom. \textbf{37} (1996), no.~1, 9--15.

\end{thebibliography}
\end{document}